\documentclass{amsart} 

\usepackage{amsthm, amssymb, amsmath, amscd}
\usepackage{eurosym}
\usepackage{amsfonts}
\usepackage{amsmath}
\usepackage{setspace}
\usepackage{graphicx}

\usepackage[all,cmtip]{xy}

\usepackage{indentfirst}

\newtheorem{theorem}{Theorem}[section]

\newtheorem{cor}[theorem]{Corollary}

\newtheorem{lemma}[theorem]{Lemma}

\newtheorem{prop}[theorem]{Proposition}

\theoremstyle{definition}
\newtheorem{define}[theorem]{Definition}
\newtheorem{ex}[theorem]{Example}
\newtheorem{remark}[theorem]{Remark}

\newcommand{\field}[1]{\mathbb{#1}}

\newcommand{\zkz}{\field{Z}/k\field{Z}}
\newcommand{\rz}{\field{R}/\field{Z}}
\newcommand{\qz}{\field{Q}/\field{Z}}

\begin{document}
\title{$\rz$-valued index theory via geometric $K$-homology}
\author{Robin J. Deeley}
\date{\today}
\begin{abstract}
A model of $K$-homology with coefficients in a mapping cone using the framework of the geometric cycles of Baum and Douglas is developed.  In particular, this leads to a geometric realization of $K$-homology with coefficients in $\rz$.  In turn, this group is related to the relative $\eta$-invariant via index pairings.
\end{abstract}
\maketitle

\section{Introduction}
The goal of this paper is the construction of a relative group in geometric $K$-homology.  This construction is inspired by ideas of Baum and Douglas (see \cite{BD}), Karoubi (see \cite[Section 2.13]{Kar}) and Stong (see \cite[Chapter 1]{Sto}).  An application is the construction of geometric models for $K$-homology with coefficients; of particular interest is the coefficient group $\rz$.  The reader is directed to \cite{HigRoeEta} (in particular, Section 6) for a discussion on the relationship between geometric $K$-homology and $\rz$-valued index theory.  The reader should also note that the starting point for the construction considered here was \cite[Remark 6.12]{HigRoeEta}. Recently (see \cite{AAS}) Antonini, Azzali, and Skandalis have considered an operator algebraic approach to $\rz$-valued index theory.
\par
To motivate our construction, we briefly discuss various pairings between $K$-theory and $K$-homology.  Let $X$ denote a finite CW-complex, $K^*(X)$ its $K$-theory and $K_*(X)$ its $K$-homology.  The index pairing between $K$-theory and $K$-homology is defined as a map $K^p(X)\times K_p(X) \rightarrow \field{Z}$.  It is useful to have an explicit realization of this pairing depending on the specific choice of cocycles used to model $K$-theory (e.g., vector bundles, projective modules or projections) and the specific choice of cycles in $K$-homology (e.g., Baum-Douglas cycles, Kasparov cycles, or Extensions).  For example, in the context of projections and Kasparov cycles, the reader can find such a formula in \cite[Section 8.2]{HR}.  In fact, this formula (and many similar formulae) factor as follows:
\begin{equation} \label{PairIntroNoCoeff}
K^p(X) \times K_p(X) \rightarrow K_0(pt) \cong \field{Z}
\end{equation}
where the isomorphism between $K_0(pt)$ and $\field{Z}$ is given (depending on context) either by the topological or analytic index map. 
\par
It is natural to ask what the situation is for pairings in $K$-theory and $K$-homology with coefficients.  As such, let $K^*(X;\rz)$ denote the $K$-theory with coefficients in $\rz$ of $X$.  There is now an index pairing $K^p(X;\rz)\times K_p(X) \rightarrow \rz$ and (in \cite{Lot}) Lott showed that it is realized analytically by the relative $\eta$-invariant.  Needless to say, the relative $\eta$-invariant is an important invariant which has been studied extensively since being introduced by Atiyah, Patodi, and Singer in \cite{APS1,APS2,APS3}.  Based on the discussion in the previous paragraph (in particular, Equation \ref{PairIntroNoCoeff}), one would like to write this pairing in the following form:
\begin{equation}
K^p(X;\rz) \times K_p(X) \rightarrow K_0(pt;\rz) \cong \rz
\end{equation}        
Thus, to obtain a suitable {\it geometric} realization of this pairing, we need to understand $K_0(pt;\rz)$ geometrically. 
\par
Analytically, $K_*(X;\rz)$ is given by $KK^*(C(X),SC_{\phi})$ where $\phi:\field{C} \hookrightarrow N$, $N$ is a ${\rm II}_1$-factor, and $C_{\phi}$ denotes the mapping cone of $\phi$.  More generally, we let $\phi:B_1\rightarrow B_2$ be any unital $*$-homomorphism between (unital) $C^*$-algebras, $B_1$ and $B_2$.  Our goal is the construction of a geometric model for $K$-homology with coefficients in the mapping cone of $\phi$.  In other words, a geometric realization of the $KK$-theory group: $KK(C(X),C_{\phi})$ where $C_{\phi}$ is the mapping cone of $\phi$.  The connection between this framework and $K$-homology with coefficients is discussed in detail in Example \ref{coeExa}. 
\par
Additional notation is required to review the Baum-Douglas model of $K$-homology and its generalization to $K$-homology with coefficients in a unital $C^*$-algebra (i.e., a geometric model for $KK(C(X),B)$ where $B$ is a unital $C^*$-algebra).  Recall that $X$ denotes a finite CW-complex.  The $C^*$-algebra of continuous $B$-valued functions on $X$ is denoted by $C(X,B)$ and the Grothendieck group of (isomorphism classes of) finitely generated projective Hilbert $B$-module bundles over $X$ is denoted by $K^0(X;B)$.    It is well-known (for example, see \cite[Proposition 2.17]{Sch}) that 
$$K^0(X;B) \cong K_0(C(X,B))\cong K_0(C(X)\otimes B)$$  
We will refer to finitely generated projective Hilbert $B$-modules bundles simply as $B$-module bundles; all bundles in this paper are assumed to be locally trivial.  The reader should recall (see \cite{BD, BD2, BDrel, BDT, BHS}) that a cycle (over $X$) in the Baum-Douglas model (for $K$-homology) is given by a triple, $(M,E,f)$, where $M$ is a smooth compact $spin^c$-manifold, $E$ is a smooth Hermitian vector bundle over $M$, and $f:M\rightarrow X$ is a continuous map.  Addition of cycles is defined using disjoint union.  The geometric $K$-homology group (denote $K_*(X)$) is given by equivalence classes of cycles under the relation generated by three elementary operations: disjoint union/direct sum, bordism and vector bundle modification.  The reader can find a nice treatment of the construction of this model in \cite{BHS}.    
\par
This model has been generalized to produce models for $KK(C(X),B)$ (recall that $B$ denotes a unital $C^*$-algebra); one replaces the vector bundle, $E$, with a (finitely generated projective) Hilbert $B$-module bundle.  The operations, relations, and map to $KK(C(X),B)$ are analogous to those on the original cycles of Baum and Douglas; we denote the (geometric) group defined via these cycles and relation by $K_*(X;B)$.  If $\phi:B_1 \rightarrow B_2$ is as above, then there is an induced map 
$$\phi_*:K_*(X;B_1) \rightarrow K_*(X;B_2)$$
which is defined at the level of geometric cycles via $(M,E_{B_1},f)\mapsto (M,E_{B_1}\otimes_{\phi} B_2,f)$.  More details on this model can be found in \cite{Wal}.
\par
Analytically (i.e., in the framework of $C^*$-algebras), the mapping cone of $\phi$ (denoted $C_{\phi}$) leads to the following six term exact sequence:
\begin{center}
$\begin{CD}
KK^0(C(X),B_1) @>\phi_*>> KK^0(C(X),B_2) @>>> KK^0(C(X),C_{\phi}) \\
@AAA @. @VVV \\
KK^1(C(X),C_{\phi}) @<<<  KK^1(C(X),B_2) @<\phi_*<< KK^1(C(X),B_1) 
\end{CD}$
\end{center}
\par
Two geometric models for $KK(C(X),C_{\phi})$ are developed (see Sections \ref{KarModel} and \ref{BorModel}).  A key part of both is the construction of an exact sequence of the same form as the one produced by the mapping cone construction (see Theorems \ref{exactSeq111} and \ref{bockTypeSeq}).  
\par
The first model is defined using cycles of the form $(M,[(E_{B_1},F_{B_1},\varphi)],f)$ where $M$ and $f$ are as in the Baum-Douglas model and $[E_{B_1},F_{B_1},\varphi]$ is an element in $K_0(C(M)\otimes C_{\phi})$.  We denote the associated $K$-homology group by $\bar{K}_*(X;\phi)$.  While these cycles are defined in a natural way, it is unclear (to the author) how to {\it naturally} incorporate certain constructions from geometric $K$-homology into this framework - in particular, construction related to the bordism relation.  For example, suppose $(M,F_{B_1},f)$ is a geometric cycle in $KK(C(X),B_1)$ whose image under $\phi_*$ is a boundary;  that is, $(M,F_{B_1}\otimes_{\phi}B_2,f)=\partial (W,E_{B_2},f)$.  Based on the exact sequence in Theorem \ref{exactSeq111}, we have a cycle $(\hat{M},[\hat{E}_{B_1},\hat{F}_{B_1},\varphi],g)$ such that 
$$(M,F_{B_1},f)\sim (\hat{M},\hat{E}_{B_1},g)\dot{\cup}(-\hat{M},\hat{F}_{B_1},g)$$  
(where $``-"$ denotes taking the opposite $spin^c$-structure).  However, it is rather difficult to {\it explicitly} find such a cycle.
\par
This consideration leads us to a second model which is defined using cycles more closely related to the bordism relation in geometric $K$-homology.  The cycles in this model are given by $(W,(E_{B_2},F_{B_1},\alpha),f)$ where $W$ is a smooth compact $spin^c$-manifold with boundary, $E_{B_2}$ is a Hilbert $B_2$-module bundle over $W$, $F_{B_1}$ is a Hilbert $B_1$-module bundle over $\partial W$, $\alpha:E_{B_2}|_{\partial W} \cong F_{B_1}\otimes_{\phi} B_2$, and $f:W\rightarrow X$ is a continuous map.  We denote the associated $K$-homology group by $K_*(X;\phi)$.  It should be evident to the reader that the problem discussed in the previous paragraph becomes a non-issue for this model.  Of course, the two models lead to isomorphic theories; we discuss an explicit isomorphism in Section \ref{isoGeoModel1And2}. 
\par
A summary of the content of the paper seems in order.  In Section 2, we discuss the $K$-theory data required for each of our models.  In Sections 3 and 4, the two geometric models for $KK(C(X),C_{\phi})$ are defined.  An isomorphism between these two theories is given in Section 5.  Finally, in Section 6, we specialize to the case of $\phi:\field{C} \rightarrow N$ (where $N$ is a ${\rm II}_1$-factor).  As noted above, this special case gives a geometric model for $K$-homology with coefficients in $\rz$.  
\par
We have assumed that the reader is familiar with the Baum-Douglas model for $K$-homology, the basic properties of Hilbert $C^*$-module bundles, and some basic $KK$-theory.  The notation of \cite{BHS}, \cite{Rav}, and \cite{Sch} (in particular, see Section 2 of this paper) are all used in this paper.  Details on the ``straightening the angle" technique can be found in \cite{CFPerMap} and \cite[Appendix]{Rav}.  In Section \ref{rzIndexSec}, a construction which is similar to one in \cite[Section 5]{APS3} is discussed.  This construction also uses the geometric model for $K_*(X;\zkz)$ discussed in \cite{Dee1} and \cite{Dee2}.    
\par
The reader will notice that the constructions in this paper are almost completely geometric (an exception is Section \ref{isoKarTypSec}).  Since the heart of index theory is the interaction between analysis and geometry, they may wonder if there is an analytic side to the constructions considered here.  This is indeed the case; a detailed development is given in \cite{DeeMapCone}.  Briefly, this work involves two main themes: 
\begin{enumerate}
\item The construction of an explicit (i.e., defined at the level of cycle) isomorphism from $K_*(X;\phi)$ to $KK^{*+1}(C(X),C_{\phi})$;
\item Connecting the index map $K_0(pt;\rz) \rightarrow \rz$ to higher APS-index theory (see \cite{PS} and references therein).
\end{enumerate}

\section{Some remarks on $K$-theory}
\subsection{$K$-theory classes for the Karoubi type model} \label{KthKarRel}
To begin, we recall the construction of a relative $K$-theory group in \cite[Section 2.13]{Kar}.  In our context, this construction leads to a realization of the $K$-theory of a certain mapping cone (see Proposition \ref{sixTermExactForRelKth} below).  As above, $\phi:B_1 \rightarrow B_2$ is a unital $*$-homomorphism and $X$ is a finite CW-complex. We apply Karoubi's relative K-theory construction in the context of the $K$-theory groups $K_*(C(X)\otimes B_i)=K^*(X;B_i)$ as realized using $B_i$-module bundles (in the case of degree zero) and unitaries (in the case of degree one); the map between these groups is given by $\phi_*: K_*(X;B_1) \rightarrow K_*(X;B_2)$.
\par 
Let $\Gamma(X;\phi)$ be the set of cocycles, $(E,F,\varphi)$, where $E$ and $F$ are $B_1$-module bundles over $X$ and $\varphi: E\otimes_{\phi}B_2 \rightarrow F\otimes_{\phi}B_2$ is an isomorphism of $B_2$-module bundles.  Two cocycles $(E,F,\varphi)$ and $(E^{\prime},F^{\prime},\varphi^{\prime})$ are isomorphic if there exist isomorphisms of $B_1$-module bundles, $\beta_1:E \rightarrow E^{\prime}$ and $\beta_2:F\rightarrow F^{\prime}$, which fit into the following commutative diagram:
\begin{center}
$\begin{CD}
E\otimes_{\phi}B_2 @>\varphi>> F\otimes_{\phi}B_2 \\
@V \phi^*(\beta_1)VV  @VV\phi^*(\beta_2)V \\
E^{\prime}\otimes_{\phi}B_2 @>\varphi^{\prime}>> F^{\prime}\otimes_{\phi}B_2
\end{CD}$
\end{center}
\vspace{0.2cm}
A cocyle, $(E,F,\varphi)$, is elementary if $E=F$ and $\varphi$ is homotopic to $Id_{E\otimes_{\phi}B_2}$.  We can add cocycles using direct sum.  
\begin{define}
Let $K^0(X;\phi)$ be the quotient of $\Gamma(X;\phi)$ by the equivalence relation $\epsilon \sim \epsilon^{\prime}$ if there exists elementary cocycles $\varepsilon$ and $\varepsilon^{\prime}$ such that $\epsilon+\varepsilon \cong \epsilon^{\prime}+\varepsilon^{\prime}$. 
\end{define}
There is a similar definition of $K^1(X;\phi)$ (see \cite{Kar} or \cite[Section 2.3]{AAS}); however, we will (apart from the next proposition) only need $K^0(X;\phi)$. The next proposition summarizes the basic properties of $K^*(X;\phi)$ (see \cite[Chapter 2]{Kar} for details). The reader can find an expicit isomorphism from $K^*(X;\phi)$ to the K-theory of the mapping cone in \cite[Section 2.3]{AAS}.  
\begin{prop} \label{sixTermExactForRelKth}
$K^*(X;\phi)$ is an abelian group which is naturally isomorphic to $K_*(C(X)\otimes C_{\phi})$ where $C_{\phi}$ is the mapping cone of $\phi$.  That is, 
$$C_{\phi}=\{ (f,b_1)\in C([0,1),B_2)\oplus B_1 \: | \: f(0)=\phi(b_1)\} $$
In particular, Bott periodicity leads to the following six term exact sequence \\
\begin{center}
$\begin{CD}
K_0(C(X)\otimes B_1) @>\phi_*>> K_0(C(X)\otimes B_2) @>\hat{r}>> K^1(X;\phi) \\
@AA\partial A @. @VV\partial V \\
K^0(X;\phi) @<\hat{r}<<  K_1(C(X)\otimes B_2) @<\phi_*<< K_1(C(X)\otimes B_1) 
\end{CD}$
\end{center}
The maps in this exact sequence are defined explicitly at the level of cocycles.  For example, in the case of $K^0(X;\phi)$, if $u$ is a unitary in $M_n(B_2)$ and $(E,F,\varphi)$ is a cocycle in $K^0(X;\phi)$, then 
$$ \hat{r}([u])=[(X\times B_1^n,X\times B_1^n,u)] \hbox{ and } \partial [E,F,\varphi]=[E]-[F]$$
\end{prop}
\begin{remark} \label{relKthIsLikeKth}
It follows from this proposition that $K^*(X;\phi)$ has many properties in common with $K$-theory.  For example, this theory has a version of the Thom isomorphism, is a module over $K$-theory, and if $M$ is a compact $spin^c$-manifold then $K^*(M;\phi)$ satisfies a form of Poincare duality; namely, 
$$K^*(M;\phi) \cong KK^*(\field{C},C(M)\otimes C_{\phi}) \cong_{PD_M} KK^{*+dim(M)}(C(M),C_{\phi})$$
\end{remark}

\subsection{$K$-theory classes for the bordism type model}
In this section, we discuss the $K$-theory construction relevant for the second model for $KK(C(X),C_{\phi})$.  The results from this section will not be used until Section \ref{BorModel}.
\begin{define}
Let $W$ be a compact manifold with boundary (its boundary is denoted by $\partial W$) and $\phi:B_1\rightarrow B_2$ be a unital $*$-homomorphism.  Then 
$$C^*(W,\partial W;\phi):= \{ (f,g) \in C(W,B_2) \oplus C(\partial Q,B_1) \: | \: f|_{\partial Q} = \phi \circ g \}$$   
\end{define}
The reader should note that $C^*(W,\partial W;\phi)$ is a $C^*$-algebra and fits into the following pullback diagram:
\begin{center}
$\begin{CD}
C^*(W,\partial W;\phi) @>>> C(\partial W,B_1) \\
@VVV  @V\phi_* VV \\
C(W,B_2) @>|_{\partial}>> C(\partial W,B_2) \\ 
\end{CD}$
\end{center}

\begin{define}
Let $W$ be a compact manifold with boundary.  Then
$$K^0(W,\partial W;\phi):= \hbox{{\rm Grothendieck group}}([E_{B_2},F_{B_1},\alpha])$$
where 
\begin{enumerate}
\item $E_{B_2}$ is a finitely generated projective Hilbert $B_2$-module bundle over $W$;
\item $F_{B_1}$ is a finitely generated projective Hilbert $B_1$-module bundle over $\partial W$;
\item $\alpha$ is an isomorphism from $E_{B_2}|_{\partial W}$ to $\phi_*(F_{B_1}):=F_{B_1}\otimes_{\phi} B_2$.
\item $[\;\cdot\;]$ denotes taking the isomorphism class (the definition of isomorphism is given in \cite{Kar}; it is very similar to the definition of isomorphism for the cocycles considered in the previous section.
\end{enumerate}
\end{define}
The next two propositions are standard results. A proof of the first is given in \cite[Sections 2 and 3]{MilAlgK}; the reader is directed to \cite[Theorem 21.2.3]{Bla} or \cite[Exercise 4.10.22]{HR} for the second. In regards to applying these results, the reader should note that the map $C(W,B_2) \rightarrow C(\partial W,B_2)$ is onto.
\begin{prop}Let $W$ be a compact manifold with boundary. Then,
$$K^0(W,\partial W;\phi)\cong K_0(C^*(W,\partial W;\phi))$$
\end{prop}
\begin{prop} \label{borModKthPullExtSeq}
Let $W$ be a compact manifold with boundary.  Then, the following sequence is exact: 
\begin{center}
$\begin{CD}
K_1(C^*(W,\partial W;\phi)) @>>> K_1(C(W,B_2))\oplus K_1(C(\partial W,B_1)) @>>> K_1(C(\partial W,B_2)) \\
@AAA @. @VVV \\
K_0(C(\partial W,B_2)) @<<<  K_0(C(W,B_2))\oplus K_0(C(\partial W,B_1)) @<<< K_0(C^*(W,\partial W;\phi)) 
\end{CD}$
\end{center}
\end{prop}
\section{Model for $KK(C(X),C_{\phi})$ via relative $K$-theory} \label{KarModel}
As in the introduction, $X$ denotes a finite CW-complex, $\phi:B_1 \rightarrow B_2$ denotes a unital $*$-homomorphism (between unital $C^*$-algebras $B_1$ and $B_2$), and $C_{\phi}$ denotes the mapping cone of $\phi$. We note that the $C^*$-algebra $C_{\phi}$ is not unital. Hence the geometric model for K-homology discussed in \cite{Wal} can not be applied directly to obtain a model for $KK(C(X), C_{\phi})$; the goal of this section is the construction of a geometric model for this $KK$-theory group.
\subsection{Cycles and relations} 
\begin{define} \label{relKthCycMod}
A $\bar{K}$-cycle (over $X$ with respect to $\phi$) is a triple, $(M,[(E,F,\varphi)],f)$ where 
\begin{enumerate}
\item $M$ is a smooth compact $spin^c$-manifold; 
\item $[(E,F,\varphi)]$ is a class in $K^0(M;\phi)$;
\item $f:M \rightarrow X$ is a continuous map;
\end{enumerate}
\end{define}
In this section, we will refer to $\bar{K}$-cycles simply as cycles. There is a natural definition of isomorphism for such cycles (see page 75 of \cite{Wal} for details).  When we refer to a ``cycle", we will in fact be refering to an isomorphism class of a cycle.  We can add (isomorphism classes of) cycles using the disjoint union operation; that is, 
$$(M,[(E,F,\varphi)],f)+(M^{\prime},[(E^{\prime},F^{\prime},\varphi^{\prime})],f^{\prime})= (M\dot{\cup}M^{\prime},[(E,F,\varphi)]\dot{\cup}[(E^{\prime},F^{\prime},\varphi^{\prime})],f\dot{\cup}f^{\prime}) $$
Associated to a cycle is its ``opposite" which is obtained from the same data with the $spin^c$-structure on the manifold reversed; we denote $M$ with its opposite $spin^c$-structure by $-M$.
\begin{define}
A $\bar{K}$-bordism (over $X$ with respect to $\phi$) is given by $(W,[(E,F,\varphi)],g)$, where 
\begin{enumerate}
\item $W$ is a smooth compact $spin^c$-manifold with boundary;
\item $[(E,F,\varphi)]$ is a class in $K^0(W;\phi)$;
\item $g:W\rightarrow X$ is a continuous map;
\end{enumerate}
The boundary of a $\bar{K}$-bordism is given by
$$(\partial W, [(E|_{\partial W}, F|_{\partial W}, \varphi|_{\partial W}], f|_{\partial W})$$
If $(M_1, [E_1, F_1, \varphi_1], f_1) \dot{\cup} -(M_1, [E_1, F_1, \varphi_1], f_1)$ is the boundary of a $\bar{K}$-bordism, then we write
$$(M_1, [E_1, F_1, \varphi_1], f_1) \sim_{bor} (M_1, [E_1, F_1, \varphi_1], f_1)$$
\end{define}
\noindent
Often, in particular in this section, we refer to a $\bar{K}$-bordism simply as a bordism.
\begin{prop}
Bordism is an equivalence relation on the set of cycles.
\end{prop}
\begin{proof}
The proof is very similar to the proof in \cite[Lemma 2.1.10]{Wal}) and is left to the reader.
\end{proof}
\begin{define}
Let $(M,[E,F,\varphi],f)$ be a cycle and $V$ a $spin^c$-vector bundle of even rank over $M$.  Then the vector bundle modification of $(M,[E,F,\varphi],f)$ by $V$ is defined to be:
$$(M^V,\pi^*([(E,F,\varphi)])\otimes_{\field{C}}\beta_V,f\circ \pi) $$
where
\begin{enumerate}
\item ${\bf 1}$ is the trivial real line bundle over $M$ (i.e., $M\times \field{R}$).
\item $M^V=S(V\oplus {\bf 1})$ (i.e., the sphere bundle of $V\oplus {\bf 1}$).
\item $\beta_V$ is the ``Bott element" in $K^0(M^V)$ (see Section 2.5 of \cite{Rav} for the construction of this element).
\item $\otimes_{\field{C}}$ denotes the $K^0(M^V)$-module structure of $K^0(M^V;\phi)$.
\item $\pi:M^V \rightarrow M$ is the bundle projection.
\end{enumerate}
The vector bundle modification of a cycle $(M,[E,F,\varphi],f)$ by $V$ is often denoted by $(M,[E,F,\varphi],f)^V$.
\end{define}
A cycle, $(M,[E,F,\varphi],f)$, is called even (resp. odd) if the dimensions of the connected components of $M$ are all even (resp. odd) dimensional.
\begin{define}
Let $X$ be a finite CW-complex and $\sim$ be the equivalence relation generated by bordism and vector bundle modification (i.e., $(M,[E,F,\varphi],f)^V \sim (M,[(E,F,\varphi],f)$).  Then
\begin{eqnarray*}
\bar{K}_0(X;\phi) & := & \{ \hbox{ even cycles } \}/\sim \\
\bar{K}_1(X;\phi) & := & \{ \hbox{ odd cycles } \}/\sim
\end{eqnarray*}
\end{define}
\begin{prop}
If $X$ is a finite CW-complex, then $\bar{K}_*(X;\phi)$ is a graded abelian group.
\end{prop}
\begin{proof}
The disjoint union operation gives $\bar{K}_*(X;\phi)$ the structure of an abelian semi-group.  Any cycle which is a boundary (for example, the empty cycle) gives the identity (i.e., zero) class and the opposite of a cycle gives an inverse.
\end{proof}
The functorial properties of the group $\bar{K}_*(X;\phi)$ are similar to those of $K$-homology with coefficients in a $C^*$-algebra.  For example, if $Z$ is another finite CW-complex and $g:X\rightarrow Z$ is a continuous map, then $g_*(M,[E,F,\varphi],f):=(M,[E,F,\varphi],g\circ f)$.  

\subsection{Isomorphism with analytic theory} \label{isoKarTypSec}
To begin, note that associated to a geometric cycle (as in Definition \ref{relKthCycMod}) are the following $KK$-theory classes:
\begin{enumerate}
\item $[D_M]\in KK^{{\rm dim}(M)}(C(M),\field{C})$, the class of the Dirac operator on $M$;
\item $[(E,F,\varphi)]\in KK^0(\field{C},C(M)\otimes C_{\phi})\cong K^0(M;\phi)$;
\item $[f]\in KK(C(X),C(M))$;
\end{enumerate}
Combining these classes leads to the following map:
\begin{equation} \label{defAlpha}
\alpha: (M,(E,F,\varphi),f) \mapsto [f]\otimes_{C(M)} [(E,F,\varphi)]\otimes_{C(M)} \Delta
\end{equation}
where $\Delta$ is the image of the Dirac class under the map on $K$-homology induced from the diagonal inclusion of $M$ into $M\times M$; we denote the diagonal inclusion by diag$_M$. The map in Equation \ref{defAlpha} can also be written as 
$$\alpha: (M,[E,F,\varphi],f) \mapsto f_*(PD_M([E,F,\varphi]))$$  
where $PD_M$ denotes the Poincare duality map discussed in Remark \ref{relKthIsLikeKth}.
\begin{prop}
The map $\alpha: \bar{K}_*(X;\phi) \rightarrow KK^*(C(X),C_{\phi})$ (as defined in Equation \ref{defAlpha}) is well-defined.
\end{prop}  
\begin{proof}
We begin with the bordism relation. Let $(W,[(\tilde{E},\tilde{F},\tilde{\varphi})],\tilde{f})$ be a bordism with boundary $(M,[(E,F,\varphi)],f)$. Let $\partial$ denote the boundary map associated to the six-term exact sequence in analytic $K$-homology associated to the short exact sequence:
$$0 \rightarrow C_0(W) \rightarrow C(W) \rightarrow C(\partial W) \rightarrow 0$$
It is well known (see for example \cite[Theorem 3.5 iii)]{BHS}) that $\partial [D_W]=[D_M]$.  In addition, $f=\tilde{f}\circ r$ where $r:M\hookrightarrow W$ is the inclusion of the boundary.  Using the fact that $[r]\otimes [\partial]=0$ and basic properties of $KK$-theory, we obtain 
\begin{eqnarray*}
\alpha(M,[(E,F,\varphi)],f) & = & [f]\otimes_{C(M)} [(E,F,\varphi)]\otimes_{C(M)} [{\rm diag}_M] \otimes_{C(M)} [D_M]\\ 
& = & [\tilde{f}\circ r] \otimes_{C(M)} [E,F,\varphi]\otimes_{C(M)} [{\rm diag}_M] \otimes [\partial] \otimes [D_W] \\
& = & [\tilde{f}]\otimes [r] \otimes [\partial] \otimes [\tilde{E},\tilde{F},\tilde{\varphi}]\otimes [{\rm diag}_W] \otimes [D_W] \\
& = & 0
\end{eqnarray*}
Next, the case of vector bundle modification is considered.  Let $(M,[E,F,\varphi],f)$ be a cycle and $V$ a $spin^c$-vector bundle over $M$ with even-dimensional fibers.  Let $\pi:M^V\rightarrow M$ denote the projection map and $s:M\rightarrow M^V$ the inclusion of $M$ via the ``north pole".  Using standard results on wrong-way maps (see for example \cite{BOOSW}), we have that $s!=PD^{-1}_{M^V}s_* PD_{M}$.  All of this, along with the fact that $\pi \circ s=id_M$, leads to 
\begin{eqnarray*}
\alpha ((M,[E,F,\varphi],f)^V) & = & \alpha(M^V, s!([E,F,\varphi],f\circ \pi) \\  
& = & (f_* \circ \pi_*) PD_{M^V} s! ([E,F,\varphi]) \\
& = & f_* \circ ( \pi_* \circ s_*) PD_M ([E,F,\varphi]) \\
& = & f_* PD_M ([E,F,\varphi]) \\
& = & \alpha ((M,[E,F,\varphi],f))
\end{eqnarray*}
\end{proof}
Following \cite{BOOSW}, we introduce conditions which allow for the construction of an inverse to the map $\alpha$.  Namely, suppose that there exists compact $spin^c$-manifold $Z$ and continuous maps $h:X\rightarrow Z$ and $g:Z \rightarrow X$ such that $g\circ h$ is homotopic to the identity map on $X$.  We note that each finite CW-complex satisfies this condition (see \cite[Lemma 2.1]{BOOSW} for details).  The map from $KK^*(C(X),C_{\phi})\rightarrow \bar{K}_{*+1}(X;\phi)$ is defined at the level of cycles via
$$\beta: \xi \in KK(C(X),SC_{\phi}) \mapsto (Z, PD_{Z}^{-1}(h_*(\xi)) ,g)$$
Also, given a (smooth) emedding of $spin^c$-manifolds, let $g!$ denote the wrong-way map associated to $g$; the precise definition of this map can be found (for example) in \cite[Section 1.15]{Wal}. 
\begin{lemma} \label{embVBMlemForIso}
Suppose $g:N\rightarrow M$ is an embedding of compact $spin^c$-manifolds with the codimension of $g(N)$ even-dimensional.  Also, let $(M,[E,F,\varphi],f)$ be a cycle (in $\bar{K}_*(X;\phi)$).  Then
$$(M, g![(E,F,\varphi)],f) 
 \sim  (N,[(E,F,\varphi)],f\circ g)$$
where $g!$ denotes the wrong way map associated to the embedding $g$.
\end{lemma}
\begin{proof}
The proof is essentially the same as the proof of Lemma 2.3.4 in \cite{Wal}; the details are left to the reader.  
\end{proof}
\begin{lemma}
Let $X$ be a compact Hausdorff space.  Then $\alpha \circ \beta = id$.  Moreover, if $X$ is a smooth compact $spin^c$-manifold and we take $Z=X$, $g=h=id$, then $\beta \circ \alpha =id$. 
\end{lemma}
\begin{proof}
The first equality follows from
\begin{eqnarray*}
\alpha(\beta(\xi)) & = & \alpha(Z, PD^{-1}_M(h^*(\xi)) , g) \\
& = & g_*((PD_M \circ PD^{-1}_M \circ h_*)(\xi))   \\
& = & (g_* \circ h_*)(\xi) \\
& = & \xi
\end{eqnarray*}
For the second equality, we have
\begin{eqnarray*}
\beta(\alpha(M,[(E,F,\varphi)],f)) & = & (X, (PD^{-1}_X \circ f_* \circ PD_X)([E,F,\varphi]),id_X) \\
& = & (X, f![(E,F,\varphi)],id_X) 
\end{eqnarray*}
Thus, the proof is reduced to showing that $(X, f![(E,F,\varphi)],id_X) \sim (M,[(E,F,\varphi)],f)$.  We would like to use Lemma \ref{embVBMlemForIso}.  However, in general, $f$ need not be an embedding.  To circumvent this problem, let $e:M \hookrightarrow S$ be an embedding of $M$ into a sphere of even (resp. odd) dimension if $M$ is even (resp. odd) dimensional.  Then 
$$ (f,e): M \hookrightarrow X\times S \hbox{ and } (id_X,0): X \hookrightarrow X\times S$$
are embeddings.  Moreover, $(f,e)$ is homotopic to $(f,0)$.  Using Lemma \ref{embVBMlemForIso} (twice), we obtain
\begin{eqnarray*}
(M,E,f) & \sim & (X\times S, (f,e)![E,F,\varphi],proj_X) \\
& \sim & (X\times S,(f,0)![E,F,\varphi],proj_X) \\
& \sim & (X,f![E,F,\varphi],id_X) 
\end{eqnarray*}
and hence the desired result.  
\end{proof}

\begin{theorem} \label{isoRelKtheoryCase}
If $X$ is a finite CW-complex, then $\alpha$ is an isomorphism.
\end{theorem}
\begin{proof}
The previous lemma implies that $\alpha$ has a right inverse (namely, $\beta$).  To see that $\beta$ is also a left inverse, consider the following commutative diagram:
\begin{center}
$\begin{CD}
\bar{K}_*(X;\phi) @>h_*>> \bar{K}_*(Z; \phi) \\
@VV\beta \circ \alpha V  @VV id V \\
\bar{K}_*(X;\phi) @>h_*>> \bar{K}_*(Z; \phi) 
\end{CD}$
\end{center}
\vspace{0.2cm}
where we have used the previous lemma to obtain that $\beta \circ \alpha : \bar{K}_*(Z;\phi) \rightarrow \bar{K}_*(Z ;\phi)$ is equal to the identity map.  It follows (using the fact that $h_*$ is injective) that $\beta \circ \alpha: \bar{K}_*(X;\phi) \rightarrow \bar{K}_*(X;\phi)$ is the restriction of the identity and hence {\it is} the identity map.
\end{proof}
The reader should recall (from the introduction) that if $B$ is a unital $C^*$-algebra, then $K_*(X;B)$ denotes the Baum-Douglas model of $KK^*(C(X),B)$ (see \cite{Wal} for details).
\begin{theorem} \label{exactSeq111}
If $X$ is a finite CW-complex, then the following sequence is exact:
\\
\begin{center}
$\begin{CD}
K_0(X;B_1) @>\phi_*>> K_0(X;B_2) @>\bar{r}>> \bar{K}_1(X;\phi) \\
@AA\bar{\delta} A @. @VV\bar{\delta} V \\
\bar{K}_0(X;\phi) @<\bar{r}<<  K_1(X;B_2) @<\phi_*<< K_1(X;B_1) 
\end{CD}$
\end{center}
\vspace{0.2cm}
where the maps are defined as follows:
\begin{enumerate}
\item $\phi_* : K_*(X;B_1) \rightarrow K_*(X;B_2)$ takes a cycle $(M,F,f)$ to $(M,\phi_*(F),f)=(M, F\otimes_{\phi}B_2, f)$. 
\item $\bar{r}: K_*(X;B_2) \rightarrow \bar{K}_{*+1}(X;\phi)$ takes a cycle $(M,E,f)$ to $(M\times S^1,\hat{r}(u),f\circ \pi)$ where 
\begin{eqnarray*}
u:\pi^*(E) & \rightarrow &  \pi^*(E) \\
 ((m,z), e) & \mapsto & ((m,z),z\cdot e), 
\end{eqnarray*}  
$\hat{r}$ is the map in the exact sequence in the statement of Proposition \ref{sixTermExactForRelKth}, and $\pi:M\times S^1 \rightarrow M$ is the projection map.
\item $\bar{\delta} : \bar{K}_*(X;\phi) \rightarrow K_*(X;B_1)$ takes a cycle $(M,[(E,F,\alpha)],f)$ to $(M,[E]-[F],f)$.  
\end{enumerate}
\end{theorem}
\begin{proof}
Theorem \ref{isoRelKtheoryCase} and \cite[Theorem 2.3.3]{Wal} imply that
$$K_*(X;B_i) \cong KK^*(C(X),B_i) \hbox{ and } \bar{K}_*(X;\phi)\cong KK^*(C(X),C_{\phi})$$
via the explicit maps 
$$ \mu_{B_i}:K_*(X;B_i) \rightarrow KK^*(C(X),B_i) \hbox{ and } \alpha: \bar{K}_*(X;\phi) \rightarrow KK^*(C(X),C_{\phi}) $$
This reduces the proof to showing that the following diagram commutes:
{\footnotesize \begin{center}
$\minCDarrowwidth10pt\begin{CD}
@>>> K_*(X;B_1) @>\phi_*>> K_*(X;B_2) @>\bar{r}>> \bar{K}_{*+1}(X;\phi) @>\bar{\delta}>> K_{*+1}(X;B_1) @>>> \\
@. @VV\mu_{B_1} V @VV\mu_{B_2} V @VV\alpha V @VV\mu_{B_1}V @.  \\
@>>> KK^*(C(X),B_1) @>\phi_*>> KK^*(C(X),B_2) @>>> KK^{*+1}(C(X),C_{\phi}) @>>> KK^{*+1}(C(X),B_1) @>>>
\end{CD}$
\end{center}}
\vspace{0.2cm}
\noindent
where the bottom long exact sequence is obtained from the following short exact sequence of $C^*$-algebras (see \cite[Section 15.3]{Bla}):
$$0\rightarrow SB_2 \stackrel{\iota}{\rightarrow} C_{\phi} \stackrel{ev_0}{\rightarrow} B_1 \rightarrow 0$$
That $\phi_* \circ \mu_{B_1}=\phi_* \circ \mu_{B_2}$ follows from \cite[Proposition 2.2.10]{Wal}.  The proof that $\mu_{B_1}\circ \bar{\delta} =(ev_0)_* \circ \alpha$ is as follows. Let $\Delta$ denote the class in $KK^{{\rm dim}(M)}(C(M)\otimes C(M), \field{C})$ which implements the map $PD_M$; that is, $PD_M : KK^*(\field{C}, C(M)) \rightarrow KK^{*+{\rm dim}(M)}(C(M), \field{C})$ is given by $x \mapsto x \otimes_{C(M)} \Delta$. Denoting the classes in $KK$-theory associated to $f$ and $ev_0$ respectively by $[f]$ and $[ev_0]$ and using basic properties of $KK$-theory, we have that
\begin{eqnarray*}
((ev_0)_* \circ \alpha)(M,[E,F,\varphi],f) & = & (ev_0)_* (f_*(PD_M([E,F,\varphi]))) \\
& = & [f]\otimes_{C(M)} ( [E,F, \varphi] \otimes_{C(M)} \Delta ) \otimes_{C_{\phi}} [ev_0] \\
& = & [f] \otimes_{C(M)} ( [E,F, \varphi] \otimes_{C(M)\otimes C_{\phi}} (\Delta \otimes_{\field{C}} [ev_0])) \\
& = & [f] \otimes_{C(M)} ( [E,F,\varphi] \otimes_{C(M)\otimes C_{\phi}} ([ev_0] \otimes_{\field{C}} \Delta)) \\
& = & f_*(PD_M((ev_0)_*([E,F,\varphi]))) \\
 & = & \mu_{B_1} (M, (ev_0)_* [E,F,\varphi],f) \\ 
& = & (\mu_{B_1} \circ \bar{\delta})(M,[E,F,\varphi],f)
\end{eqnarray*}   
Finally, we must show that $ \iota_* \circ \mu_{B_2}=\alpha \circ \bar{r}$.  Let $(M,E_{B_2},f)$ be a cycle in $K_*(X;B_2)$.  Then, again using the fact the $PD_M$ is implemented via an explicit KK-class and basic properties of $KK$-theory, we obtain
\begin{eqnarray*}
(\iota_* \circ \mu_{B_2})(M,E_{B_2},f) & = & \iota_* ( f_*(PD_M([E_{B_2}]))) \\
& = & \iota_* ( (f \circ \pi_{S^1})_*(PD_{M\times S^1} (\pi^*_{S^1}(E_{B_2})\otimes u))) \\
& = & (f_* \circ \pi_{S^1})_*(PD_{M\times S^1} (\hat{r}(\pi^*_{S^1}(E_{B_2})\otimes u))) \\
& = & (\alpha \circ \bar{r})(M,E_{B_2},f)
\end{eqnarray*}
where the reader should note that $u$ is defined in the statement of the theorem.
\end{proof}

\section{Bordism type model for $KK(C(X),C_{\phi})$}  \label{BorModel}
\subsection{Cycles and Relations}
Based on the discussion in the introduction, we need different geometric cycles to model certain constructions in geometric $K$-homology.  As such, in this section, we introduce an abelian group, $K_*(X;\phi)$, with the main result that this group fits into a six-term exact sequence (see Theorem \ref{bockTypeSeq}).  In fact, we define two types of cycles.  The first uses vector bundle data while the second uses $K$-theory data; a detailed development is only given in the case of the second model as the two cases are rather similar.  
\begin{define} {\bf Cycles with vector bundle data} \label{cycBunDat} \\ 
A cycle (over $X$ with respect to $\phi$ with bundle data) is given by, $(W, (E_{B_2},F_{B_1},\alpha),f)$, where 
\begin{enumerate}
\item $W$ is a smooth, compact $spin^c$-manifold with boundary;
\item $E_{B_2}$ is a finitely generated projective Hilbert $B_2$-module bundle over $W$;
\item $F_{B_1}$ is a finitely generated projective Hilbert $B_1$-module bundle over $\partial W$;
\item $\alpha : E_{B_2}|_{\partial W} \cong \phi_*(F_{B_1})$ is an isomorphism of Hilbert $B_2$-module bundles;
\item $f:W \rightarrow X$ is a continuous map.
\end{enumerate}
\end{define}
\begin{define} {\bf Cycles with $K$-theory data} \label{phiCyc} \\
A cycle (over $X$ with respect to $\phi$ with $K$-theory data) is a triple, $(W,\xi,f)$, where:
\begin{enumerate}
\item $W$ is a smooth, compact $spin^c$-manifold with boundary;
\item $\xi \in K^0(W,\partial W;\phi)$;
\item $f:W \rightarrow X$ is a continuous map.
\end{enumerate}
\end{define}
Often, in particular for the rest of this section, a cycle will refer to a cycle over $X$ with respect to $\phi$ with $K$-theory data; we also refer to ``cycles" in $K_*(X;B_i)$, but this should cause no confusion. The manifold, $W$, in a cycle need not be connected.  As such, a cycle is called even (resp. odd) if each of its connected components are even (resp. odd) dimensional.  We also let $\xi_{\partial W}$ and $\xi_{W}$ denote the images of $\xi$ under the maps $p_1:K^0(W,\partial W;\phi) \rightarrow K^0(\partial W;B_1)$ and $p_2:K^0(W,\partial W;\phi) \rightarrow K^0(W;B_2)$ respectively.  The opposite of a cycle, $(W,\xi,f)$, is the same data only $W$ is given the opposite $spin^c$-structure.  It is denoted by $-(W,\xi,f)$.  The disjoint union of cycles, $(W,\xi,f)$ and $(\tilde{W},\tilde{\xi},\tilde{f})$, is given by the cycle:
$$(W\dot{\cup}\tilde{W},\xi \dot{\cup} \tilde{\xi},f\dot{\cup} \tilde{f}) $$  
Two cycles, $(W,\xi,f)$ and $(\tilde{W},\tilde{\xi},\tilde{f})$, are isomorphic if there exists a diffeomorphism, $h:W \rightarrow \tilde{W}$, such that $h$ preserves the $spin^c$-structure, $h^*(\tilde{\xi})=\xi$, and $\tilde{f} \circ h=f$.  Throughout, a ``cycle" more precisely refers to an isomorphism class of a cycle.
\begin{define}
A regular domain, $Y$, of a manifold $M$ is a closed submanifold of $M$ such that
\begin{enumerate}
\item ${\rm int}(Y)\ne \emptyset$;
\item If $p\in \partial Y$, then there exists a coordinate chart, $\phi:U \rightarrow \field{R}^n$ centered at $p$ such that $\phi(Y \cap U)=\{ x \in \phi(U) \: | \: x_n\ge 0 \} $. 
\end{enumerate}
\end{define}
\begin{define}
A bordism (with respect to $X$ and $\phi$ with $K$-theory data) is given by $(Z,W,\eta,F)$ where 
\begin{enumerate}
\item $Z$ is a compact $spin^c$-manifold;
\item $W \subseteq \partial Z$ is a regular domain;
\item $\eta \in K^0(Z,\partial Z - {\rm int}(W);\phi)$;
\item $g:Z \rightarrow X$ is a continuous map.  
\end{enumerate}
\label{borBou}
\end{define}
\begin{remark} \label{phiBor}
The ``boundary" of a bordism, $(Z,W,\eta,F)$, is given by $(W,\eta|_W,g|_W)$.  The fact that $W$ is a regular domain of $\partial Z$ ensures the boundary is indeed a cycle (as in Definition \ref{phiCyc}).  Moreover, if $(W,\xi,f)$ is a boundary in the sense of Definition \ref{borBou}, then $(\partial W, \xi_{B_1},f|_{\partial W})$ is a boundary as a cycle in $K_*(X;B_1)$.  
\end{remark}
Often, we will refer to a bordism with respect to $X$ and $\phi$ with $K$-theory data simply as a bordism; the context should make it clear with respect to which type of cycles (e.g., $K_*(X;B)$, $\bar{K}(X;\phi)$, etc) the bordism is related.
\begin{define}
Two cycles are bordant if there exists bordism with boundary given by $(W,\xi,f) \dot{\cup} -(W^{\prime},\xi^{\prime},f^{\prime})$.  This relation is denoted by 
$$(W,\xi,f) \sim_{bor} (W^{\prime},\xi^{\prime},f^{\prime})$$
The similarity between the notation of this definition and that of the definition of the bordism relation in the previous section should cause no confusion.
\end{define}
\begin{prop} \label{borRelMod2IsEqu}
The relation $\sim_{bor}$ is an equivalence relation.
\end{prop}
\begin{proof}
We begin with reflexivity.  Results in \cite[Sections I.3 and I.4]{CFPerMap}  imply that there exists a smooth $spin^c$-manifold with boundary, $Z$, which is homeomorphic (via $h$) to $W\times [0,1]$.  Moreover, $W \dot{\cup} -W$ is a regular domain of the boundary of this manifold.  Thus, if $\pi_W:W\times [0,1]\rightarrow W$ is the projection onto $W$, then 
$$(Z, W\dot{\cup} -W, h^*(\pi_W^*(\xi)) , f\circ \pi_W \circ h) $$
forms a bordism.  Moreover, it has boundary (in the sense of Definition \ref{borBou}) $(W,\xi,f) \dot{\cup} -(W,\xi,f)$.  In other words,
$$(W,\xi,f)\sim_{bor} (W,\xi,f)$$
\par
That bordism is symmetric is trivial. \par
For transitivity, let $\{(W_i,\xi_i,f_i)\}_{i=0}^{2}$ be cycles and $(Z_0,W_0 \dot{\cup} -W_1,\nu_0,F_0)$ and $(Z_1,W_1 \dot{\cup} -W_2, \nu_1,F_1)$ be bordisms from $(W_0,\xi_0,f_0)$ to $(W_1,\xi_1,f_1)$ and $(W_1,\xi_1,f_1)$ and $(W_2,\xi_2,f_2)$ respectively.  Then, by ``straightening the angle" (see \cite{CFPerMap}), $Z_0\cup_{W_1}Z_1$ can be given the structure of a smooth $spin^c$-manifold.  Also, let $F=F_0\cup_{W_1}F_1$ and $\nu=\nu_0\cup_{W_1}\nu_1$.  Then, $(Z,W_0 \dot{\cup}W_2,\nu,F)$ forms a bordism from $(W_0,\xi_0,f_0)$ to $(W_2,\xi_2,f_2)$. 
\end{proof}
\begin{define}
Let $(W,\xi,f)$ be a cycle and $E$ a $spin^c$-vector bundle of even rank over $W$.  Then the vector bundle modification of $(W,\xi,f)$ by $E$ is defined to be:
$$(W^E,\pi^*(\xi)\otimes_{\field{C}}\beta_E,f\circ \pi) $$
where
\begin{enumerate}
\item ${\bf 1}$ is the trivial real line bundle over $W$ (i.e., $W\times \field{R}$);
\item $W^E=S(E\oplus {\bf 1})$ (i.e., the sphere bundle of $E\oplus {\bf 1}$);
\item $\beta_E$ is the ``Bott element" in $K^0(W^E)$ (see \cite[Section 2.5]{Rav} for the construction of this element);
\item $\otimes_{\field{C}}$ denotes the $K^0(W^E)$-module structure of $K^0(W^E,\partial W^E;\phi)$;
\item $\pi:W^E \rightarrow W$ is the bundle projection.
\end{enumerate}
The vector bundle modification of $(W,\xi,f)$ by $E$ is often denoted by $(W,\xi,f)^E$.
\end{define}  
\begin{remark} \label{phiVBM}
If $(W,\xi,f)$ is a cycle and $E$ is a $spin^c$-vector bundle of even rank over $W$, then $(\partial W, \xi_{B_{1}},f|_{\partial W} )^{E|_{\partial W}}= \partial (W,\xi,f)^E$.
\end{remark}
\begin{define} \label{equCyc}
Let $\sim$ be the equivalence relation generated by bordisms and vector bundle modification (i.e., $(W,\xi,f)\sim (W,\xi,f)^E$, for any even rank $spin^c$-vector bundle, $E$, over $W$).  Also let
$$K_*(X;\phi)=\{ (W,\xi,f) \}/\sim$$
The grading is given as follows.  A cycle $(W,\xi,f)$ is said to be even (resp. odd) if the connected components of $W$ are all even (resp. odd) dimensional.  Then, $K_0(X;\phi)$ is even cycles modulo $\sim$ and $K_1(X;\phi)$ is likewise only with odd cycles; the relation $\sim$ preserves this grading.
\end{define}
\begin{prop}
Let $(W,\xi_1,f)$ and $(W,\xi_2,f)$ be cycles.  Then
$$(W,\xi_1,f)\dot{\cup}(W,\xi_2,f) \sim (W,\xi_1\oplus \xi_2, f)$$
\end{prop}
\begin{proof}
The proof of this proposition is similar to the proof of Proposition 4.3.2 in \cite{Rav}.  A vector bundle modification (by a trival bundle) implies that
\begin{eqnarray*}
(W,\xi_1,f)\dot{\cup}(W,\xi_2,f) & \sim & (W\times S^2,\pi^*(\xi_1)\otimes \beta,f\circ \pi) \dot{\cup} (W\times S^2,\pi^*(\xi_2)\otimes \beta,f\circ \pi) \\
(W,\xi_1\oplus \xi_2,f) & \sim & (W\times S^2,\pi^*(\xi_1 \oplus \xi_2)\otimes \beta,f\circ \pi)
\end{eqnarray*} 
where $\beta$ is the Bott element and $\pi:W\times S^2 \rightarrow W$ is the projection map.  \par
In the proof of Proposition 4.3.2 in \cite{Rav}, an explicit bordism between $(S^2,\beta)$ and $(S^2,\beta)\dot{\cup}(S^2,\beta)$ is constructed; crossing this bordism with $W$ and ``straightening the angle" leads to  
$$(W\times S^2,\pi^*(\xi_1)\otimes \beta,f\circ \pi) \dot{\cup} (W\times S^2,\pi^*(\xi_2)\otimes \beta,f\circ \pi) \sim_{bor} (W\times S^2,\pi^*(\xi_1 \oplus \xi_2)\otimes \beta,f\circ \pi)$$
and hence the desired result.
\end{proof}
\begin{prop}
The set $K_*(X;\phi)$ with the operation of disjoint union is an abelian group.  The unit is given by the trivial (i.e., empty) cycle and the inverse of a cycle is given its opposite.
\end{prop}
\begin{proof}
It is clear that disjoint union gives $K_*(X;\phi)$ the structure of an abelian semigroup.  The proof that bordisms is an equivalence relation (i.e., Theorem \ref{borRelMod2IsEqu}) implies both that $K_*(X;\phi)$ is a group and the unit and inverses are given as in the statement of the theorem.
\end{proof}

\subsection{Normal bordisms and six-term exact sequence}
\begin{define}
Let $E$ be a vector bundle. A vector bundle, $F$, is called a complementary bundle for $E$, if $E\oplus F$ is a trivial vector bundle.  If $M$ is a manifold (possibly with boundary), then a complementary bundle for $TM$ will be called a normal bundle. 
\end{define}
Given a manifold with boundary $(W,\partial W)$, we can produce a normal bundle by taking a neat embedding in $H^n:=\{ (v_1,\ldots,v_n) \in \field{R}^n \: | \: v_1 \geq 0 \}$.  
\begin{define}
Let $(W,\xi,f)$ and $(W^{\prime},\xi^{\prime},f^{\prime})$ be two cycles in $K_*(X;\phi)$.  Then, there is a normal bordism from $(W,\xi,f)$ to $(W^{\prime},\xi^{\prime},f^{\prime})$ if there exists normal bundles $N_W$, $N_{W^{\prime}}$ (over $W$ and $W^{\prime}$ respectively) such that 
$$(W,\xi,f)^{N_W} \sim_{bor} (W^{\prime},\xi^{\prime},f^{\prime})^{N_{W^{\prime}}}$$
This relation is denoted by $\sim_{nor}$.
\end{define}
The next lemma is standard (see for example \cite[Lemma 4.5.7]{Rav}), while the one following it is a natural generalization of \cite[Lemma 4.4.3]{Rav}.  The proof of the latter is left to the reader. 
\begin{lemma} \label{norStaIso}
Normal bundles are stably isomorphic (i.e., if $N_1$ and $N_2$ are normal bundles for $W$, then there exists trivial bundles, $E_1$ and $E_2$, such that $N_1 \oplus E_1 \cong N_1 \oplus E_2$).
\end{lemma}
\begin{lemma} \label{twoVecMod}
If $E_1$ and $E_2$ are $spin^c$-vector bundles with even dimensional fibers over a compact $spin^c$-manifold $W$, then, for any $\xi \in K^0(W,\partial W;\phi)$ and $f:W \rightarrow X$, 
\begin{equation*}
(W,\xi,f)^{E_1 \oplus E_2} \sim_{bor} ((W,\xi,f)^{E_1})^{p^*(E_2)}
\end{equation*}
where $p:S(E_1\oplus {\bf 1}) \rightarrow W$ is the projection map.
\end{lemma}
\begin{prop}
Let $X$ be a finite CW-complex.  Then, the relation of normal bordism of cycles is an equivalence relation.  Moreover, it is equal to the relation constructed in Definition \ref{equCyc}. 
\end{prop}
\begin{proof}
We leave it to the reader to verify that $\sim_{nor}$ is reflexive and symmetric. \par
To show $\sim_{nor}$ is transitive, let $\{(W_i,\xi_i,f_i)\}_{i=0}^{2}$ be cycles.  Moreover, assume that for each $i=0,1,2$, $N_i$ is a normal bundle for $W_i$, and $N^{\prime}_1$ is a normal bundle for $W_1$ such that
\begin{eqnarray*}
(W_0,\xi_0,f_0)^{N_0} \sim_{bor} (W_1,\xi_1,f_1)^{N_1} \\
(W_1,\xi_1,f_1)^{N^{\prime}_1} \sim_{bor} (W_2,\xi_2,f_2)^{N_2}
\end{eqnarray*}
Lemma \ref{norStaIso} implies that there are trivial bundles $\epsilon_1$ and $\epsilon^{\prime}_1$ such that 
$$N_1 \oplus \epsilon_1 \cong N^{\prime}_1 \oplus \epsilon^{\prime}_1$$
Let $\epsilon_0$ and $\epsilon_2$ be trivial bundles (over $W_0$ and $W_2$ respectively) of the same rank as $\epsilon_1$ and $\epsilon^{\prime}_1$ respectively.  Following the notation in Lemma \ref{twoVecMod} for the bundle projections and, using the fact that trivial bundles extend across bordisms, we find that
\begin{eqnarray*}
((W_0,\xi_0,f_0)^{N_0})^{p_0^*(\epsilon_0)} \sim_{bor} ((W_1,\xi_1,f_1)^{N_1})^{p_1^*(\epsilon_1)} \\
((W_1,\xi_1,f_1)^{N^{\prime}_1})^{p^*_{1^\prime}(\epsilon^{\prime}_1)} \sim_{bor} ((W_2,\xi_2,f_2)^{N_2})^{p_2^*(\epsilon_2)}
\end{eqnarray*}
Moreover, Lemma \ref{twoVecMod} and $N_1 \oplus \epsilon_1 \cong N^{\prime}_1 \oplus \epsilon^{\prime}_1$ imply that
\begin{eqnarray*}
(W_0,\xi_0,f_0)^{N_0\oplus \epsilon_0} & \sim_{bor} & (W_1,\xi_1,f_1)^{N_1\oplus \epsilon_1} \\
& \sim_{bor} & (W_1,\xi_1,f_1)^{N^{\prime}_1\oplus \epsilon_1^{\prime}} \\
& \sim_{bor} & (W_2,\xi_2,f_2)^{N_2\oplus \epsilon_2}
\end{eqnarray*}
Transitivity of $\sim_{nor}$ then follows since $N_0 \oplus \epsilon_0$ and $N_2\oplus \epsilon_2$ are normal bundles for $W_0$ and $W_2$ respectively. \par
The proof will be complete upon showing $\sim$ and $\sim_{nor}$ are the same relation.  That $\sim_{nor}$ is a weaker relation than $\sim$ is clear.  On the other hand, we must show that 
\begin{enumerate}
\item If $(W,\xi,f)$ is a cycle and $E$ is a smooth $spin^c$-vector bundle with even-dimensional fibers over $W$, then $(W,\xi,f)^E \sim_{nor} (W,\xi,f)$.
\item Any boundary in the sense of Definition \ref{borBou} also normally bounds.
\end{enumerate}
For Item 1, let $p_E:W^E \rightarrow W$ denote the projection map, $N$ a normal bundle for $W$, and $E^c$ a complement to $E$ (i.e., $E\oplus E^c$ is a trivial bundle).  Then (as the reader can verify) $p_E^*(E^c\oplus N \oplus {\bf 1})$ is a normal bundle for $W^E$.  \par
Moreover, using Lemma \ref{twoVecMod}, we obtain
$$((W,\xi,f)^E)^{p_E^*(E^c\oplus N \oplus {\bf 1})} \sim_{bor} (W,\xi,f)^{E\oplus E^c \oplus N \oplus {\bf 1}}$$
The result (i.e., Item 1 above) follows upon noticing that $E\oplus E^c \oplus N \oplus {\bf 1}$ is a normal bundle for $W$. \par
For Item 2, let $(W,\xi,f)$ be the boundary of a bordism $(Z,W,\nu,F)$.  Let $N$ be a normal bundle for $Z$.  Then 
$$\partial (Z,W,\nu,F)^{N\oplus {\bf 1}} = (W,\xi,f)^{N|_W \oplus {\bf 1}} $$
$N|_W \oplus {\bf 1}$ is a normal bundle for $W$; hence $(W,\xi,f)$ normally bounds.
\end{proof}
\begin{cor} \label{norBorToTrivial}
Let $X$ be a finite CW-complex.  Then, a cycle in $K_*(X;\phi)$ is trivial if and only if it normally bounds.
\end{cor}
\begin{theorem} \label{bockTypeSeq}
If $X$ is a finite CW-complex, then the following sequence is exact: \\
\begin{center}
$\begin{CD}
K_0(X;B_1) @>\phi_*>> K_0(X;B_2) @>r>> K_0(X;\phi) \\
@AA\delta A @. @VV\delta V \\
K_1(X;\phi) @<r<<  K_1(X;B_2) @<\phi_*<< K_1(X;B_1) 
\end{CD}$
\end{center}
where the maps are defined as follows:
\begin{enumerate}
\item $\phi_* : K_*(X;B_1) \rightarrow K_*(X;B_2)$ takes a $B_1$-cycle $(M,F,f)$ to the $B_2$-cycle $(M,\phi_*(F),f)$. 
\item $r: K_*(X;B_2) \rightarrow K_*(X;\phi)$ takes a cycle $(M,E,f)$ to $(M,(E,\emptyset, \emptyset),f)$.  
\item $\delta : K_*(X;\phi) \rightarrow K_{*+1}(X;B_1)$ takes a cycle $(W,(E,F,\alpha),f)$ to $(\partial W,F,f|_{\partial W})$.  
\end{enumerate}
\end{theorem}
\begin{proof}
In this proof, we refer to cycles in $K_*(X;B_i)$ (respectively, bordisms with respect to $K_*(X;B_i)$) as $B_i$-cycles (respectively, $B_i$-bordisms). That the maps are well-defined is clear in the case of $\phi_*$ and follows from Remarks \ref{phiBor} and \ref{phiVBM} in the case of $r$ and $\delta$ . \par
The bordism relations on the various cycles imply that the composition of successive maps is zero.  The details in the case of $r \circ \phi_*$ are as follows.  Let $(M,\xi,f)$ be a $B_1$-cycle and $\pi_M:M\times [0,1] \rightarrow M$ be the projection onto $M$.  Then $(M\times [0,1],\pi_M^*(\xi_{B_1}),f\circ \pi_M)$ is a $B_1$-bordism between $(M,\xi_{B_1},f)$ and its opposite.  Moreover, this produces a bordism with respect to $K_*(X;\phi)$ between $(M,\phi_*(\xi_{B_1}),f)$ and the empty cycle as follows.  \par
In the notation of Definition \ref{phiBor}, let
\begin{equation*}
Z  =  M\times [0,1], \:
\eta  =  \phi_*(\pi_M^*(\xi_{B_1})), \:
g  =  f\circ \pi_M  
\end{equation*} 
For the regular domain in this bordism, we take $M$.  As such, $\partial Z - {\rm int}(M)$ is $-M$ and
the ``boundary" of this bordisms is 
$$(M,\phi_*(\xi_{B_1}),f)$$
Hence $(M,\phi_*(\xi_{B_1}),f)$ is trivial in $K_*(X;\phi)$.  We leave it to the reader to show that $\phi_* \circ \delta$ and $\delta \circ r$ are both zero. 
\par
We are left to show that
\begin{equation*}
{\rm ker}(\phi_*) \subseteq {\rm im}(\delta),\:  {\rm ker}(\delta)  \subseteq {\rm im}(r),\: {\rm ker}(r) \subseteq  {\rm im}(\phi_*)
\end{equation*}
With the goal of proving ${\rm ker}(\phi_*)  \subseteq  {\rm im}(\delta)$, fix a $B_1$-cycle, $(M,\xi_{B_1},f)$, such that $(M,\phi_*(\xi_{B_1}),f)$ is trivial in $K_*(X;B_2)$.  By \cite[Corollary 4.5.16]{Rav}, there exists a normal bundle, $N$, such that 
$$(M,\phi_*(\xi_{B_1}),f)^N \cong \partial (W,\eta_{B_2},g)$$
where $(W,\eta_{B_2},g)$ is a $B_2$-bordism.  By construction, $(\eta_{B_2},\xi_{B_1})$ is an element of the pullback along the maps $K^0(X;B_2) \rightarrow K^0(Y;B_2)$ and $K^0(Y;B_1)\rightarrow K^0(Y;B_2)$.  Proposition \ref{borModKthPullExtSeq} implies that there exists $\xi \in K^0(W,\partial W;\phi)$ which maps (under the natural maps) to $(\eta_{B_2},\xi_{B_1}) \in K^0(W;B_2) \oplus K^0(\partial W;B_1)$.  In particular, $(W,\xi,g)$ forms a cycle in $K_{*-1}(X;\phi)$.  Moreover,
$$\delta (W,\xi,g) = (M,\phi_*(\xi_{B_1}),f)^N \sim (M,\phi_*(\xi_{B_1}),f)$$
\par
To show that ${\rm ker}(\delta)  \subseteq  {\rm im}(r)$, let $(W,\xi,f)$ be a cycle in $K_*(X;\phi)$ such that $(\partial W, \xi_{B_1},f|_{\partial W})$ is trivial in $K_*(X;B_1)$.  By \cite[Corollary 4.5.16]{Rav}, there exists normal bundle, $N$ over $\partial W$, such that 
$$(\partial W, \xi_{B_1}, f|_{\partial W})^N = \partial (Z,\eta,g)$$
The normal bundle $N$ may not extend to $W$.  However, by Lemma \ref{norStaIso}, if $N^{\prime}$ is a normal bundle for $W$, there exists a trivial bundle $\epsilon$ over $\partial W$ such that $N \oplus \epsilon \cong N^{\prime}|_{\partial W}$.  Hence,
\begin{eqnarray*}
\delta ((W,\xi,f)^{N^{\prime}}) & = & (\partial W, \xi_{B_1},f)^{N^{\prime}|_{\partial W}} \\
& = & (\partial W, \xi_{B_1},f)^{N\oplus \epsilon} \\
& \sim_{bor} & ((\partial W,\xi_{B_1},f)^N)^{p^*(\epsilon)} \\
& = & \partial ((Z,\eta,g)^{\epsilon_Z})
\end{eqnarray*}
where $p:\partial W^N \rightarrow \partial W$ is the projection map and $\epsilon_Z$ is the trivial vector bundle over $Z$ with fiber dimension the same as $\epsilon$.  Summarizing, from a cycle, $(W,\xi,f)$, which is in the kernel of $\delta$, we have produced an equivalent cycle (i.e., $(W,\xi,f)^{N^{\prime}}$) whose image under $\delta$ is a boundary (rather than just trivial). \par
Hence, without loss of generality, we can (and will) assume that $(\partial W, \xi_{B_1}, f|_{\partial W})$ is a boundary in $K_{*-1}(X;B_1)$.  Let $(Z,\eta,g)$ be a $B_1$-bordism such that
$$(\partial W, \xi_{B_1}, f|_{\partial W}) = \partial (Z,\eta,g) $$
Then $(Z,\phi_*(\eta),g)$ is a $B_2$-bordism.  Form the closed (smooth, $spin^c$) manifold $\tilde{W}=W \cup_{\partial W} Z$.  As the reader will note, the $K$-theory data and continuous function are compatible along $\partial W$.  Hence, we can form the $B_2$-cycle, $(\tilde{W},\xi_{B_2}\cup \phi_*(\eta) ,f\cup g)$.  \par
It remains to show that 
$$r(\tilde{W},\xi_{B_2}\cup \phi_*(\eta) ,f\cup g) \sim (W,\xi,f)$$
This result follows from the following bordism in the $K_*(X;\phi)$:
$$(\tilde{W}\times [0,1], \tilde{W}\dot{\cup} W, \tilde{\xi}, (f\cup g)\circ \pi)$$
where $\pi:\tilde{W}\times [0,1]\rightarrow \tilde{W}$ and $\tilde{\xi}$ is the element in $K^0(\tilde{W}\times [0,1],Z;\phi)$ formed from $\pi^*(\xi_{B_2}\cup \phi^*(\eta))\in K^0(\tilde{W}\times [0,1];B_2)$ and $\eta_{B_1}\in K^0(Z;B_1)$.
\par
Finally, we show that ${\rm ker}(r) \subseteq  {\rm im}(\phi_*)$.  Let $(M,\xi_{B_2},f)$ be a $B_2$-cycle, which is mapped to the trivial element in $K_*(X;\phi)$.  By Lemma \ref{norBorToTrivial}, there exists a normal bundle, $N$, such that 
$$(M,\xi_{B_2},f)^N=\partial (Z,M^N,\eta,g)$$  
Consider $(Z,\eta_{B_2},g)$ as a bordism with respect to $K_*(X;B_2)$.  It has boundary 
$$(M,\xi_{B_2},f)^N \dot{\cup} (\partial Z - M , \phi_*(\eta_{B_1}), g|_{\partial Z - M})$$
The bordism relation in $K_*(X;B_2)$ then implies that
$$(M,\xi_{B_2},f)^N \sim \phi_*((\partial Z - M , (\eta_{B_1}), g|_{\partial Z - M}))$$
This completes the proof.
\end{proof}

\section{Isomorphism from $\bar{K}_*(X;\phi)$ to $K_*(X;\phi)$} \label{isoGeoModel1And2}
\begin{define} \label{isoDefnModel1And2}
Let $\mu_{\phi}: \bar{K}_*(X;\phi) \rightarrow K_{*+1}(X;\phi)$ be the map defined at the level of cycles via
$$ (M,[E,F,\varphi],f) \mapsto (M\times [0,1], [(\mathcal{E}_{B_2}, \mathcal{F}_{B_1}, \alpha)],f\circ \pi)$$ 
where 
\begin{enumerate}
\item $\pi:M\times [0,1] \rightarrow M$ is the projection map;
\item $\mathcal{E}_{B_2}$ is the Hilbert $B_2$-module bundle, $\phi_*(\pi^*(E))$;
\item $\mathcal{F}_{B_1}$ is the Hilbert $B_1$-module bundle defined by taking its fiber at $M\times \{0\}$ to be $E$ and its fiber at $M\times \{1\}$ to be $F$ ;
\item $\alpha: \phi_*(\mathcal{F}_{B_1}) \rightarrow \mathcal{E}_{B_2}|_{M\times \{0\} \dot{\cup} M\times \{1\}}$ given by the identity on $M\times \{0\}$ and $\varphi$ on $M\times \{1\}$;
\end{enumerate}
\end{define}
Since the definition of $\mu_{\phi}$ involves the choice of cocycle, $(E,F,\varphi)$ (rather than just the class $[E,F,\varphi]$), our first goal is to show that $\mu_{\phi}$ is well-defined at the level of cycles. We must show that 
\begin{enumerate} 
\item If $(E,F,\varphi) \cong (E^{\prime},F^{\prime},\varphi)$, then $\mu_{\phi}(M,[E,F,\varphi],f)=\mu_{\phi}(M,[E^{\prime},F^{\prime},\varphi^{\prime}],f)$.
\item If $(E,F,\varphi)$ is an elementary cocycle, then $\mu_{\phi}(M,[E,F,\varphi],f)$ is trivial. 
\end{enumerate}
The first item follows from the bordism relation in $K_*(X;\phi)$.  Let $(\beta_1,\beta_2)$ be an isomorphism from $(E,F,\varphi)$ to $(E^{\prime},F^{\prime},\varphi)$; the definition of isomorphism in this context is given in Section \ref{KthKarRel}.  Also, let $\field{D}$ denote the closed unit disk and 
$$I_1=\left[0,\frac{\pi}{2}\right],\  I_2= \left[\frac{\pi}{2}, \pi \right], \ I_3= \left[\pi, \frac{3\pi}{2}\right],\ I_4=\left[\frac{3\pi}{2}, 0\right]$$
be intervals inside $\partial \field{D}=S^1$. Given an interval $I$, let $\pi_I:M\times I \rightarrow M$ denote the projection map. Consider the following bordism (with respect to $K_*(X;\phi)$):
$$\left(M\times \field{D}, M\times \left[\frac{\pi}{2},\frac{3\pi}{4}\right] \: \dot{\cup} \: M\times \left[\frac{5\pi}{4}, \frac{3\pi}{2}\right] ,(\hat{\mathcal{E}}_{B_2},\hat{\mathcal{F}}_{B_1},\hat{\alpha} ),f\circ \pi \right) $$
where
\begin{enumerate}
\item $\hat{\mathcal{E}}_{B_2}$ is formed on $M\times S^1$ by respectively clutching bundles $\pi_{I_1}^*(F)\otimes_{\phi}B_2$, $\pi_{I_2}^*(E)\otimes_{\phi}B_2$, $\pi_{I_3}^*(E^{\prime})\otimes_{\phi}B_2$, and $\pi_{I_4}^*(F^{\prime})\otimes_{\phi}B_2$ along $M \times \{\frac{\pi}{2}\}$, $M\times \{\pi\}$, $M\times \{\frac{3\pi}{2}\}$, and $M\times \{0\}$ via the isomorphisms $\varphi$, $\phi^*(\beta_1)$, $\varphi^{\prime}$, and $\phi^*(\beta_2)$.  The reader should note that this bundle extends to all of $M\times \field{D}$ since it is formed by taking the homotopy associated to the isomorphism $\phi_*(\beta_1)$ and then straightening the angle;
\item $\hat{\mathcal{F}}_{B_1}$ is given by $\pi_{I_1}^*(F) \cup_{\beta_2} \pi_{I_4}^{*}(F^{\prime})$ on $M\times [\frac{3\pi}{2},\frac{\pi}{2}]$ and $\pi_{J_1}^*(E) \cup_{\beta_1} \pi_{J_2}^*(E^{\prime})$ on $M\times [\frac{3\pi}{4},\frac{5\pi}{4}]$ where $J_1=[\frac{3\pi}{4}, \pi]$ and $J_2=[\pi,\frac{5\pi}{4}]$; 
\item $\hat{\alpha}$ is defined by $\varphi$ on fibers at $M\times \{\frac{\pi}{2}\}$, $\varphi^{\prime}$ on fibers at $M\times \{\frac{3\pi}{2}\}$ and the identity at fibers $M\times \{\frac{3\pi}{4} \}$ and $M\times \{\frac{5\pi}{4} \}$.
\end{enumerate}
As the reader can verify, this bordism (with respect to $K_*(X;\phi)$) has the required boundary.  Such a reader may wish to begin by checking the result in the case of $M=pt$.  
\par
The second item will also follow from the bordism relation (with respect to $K_*(X;\phi)$).  We form the required bordism by taking the regular domain $M\times [0,1]$ inside $\partial (M \times \field{D})=M\times S^1$.  Then
$$((M\times \field{D},M\times [0,1]), (\phi_*(\pi^*(E)),\pi^*_{[0,1]}(E),\alpha),f\circ \pi)$$ 
forms the required bordism where the construction of the bundle data is similar to the previous argument.
\par
Next, the relations must be considered.  For vector bundle modification, let $(M,[E,F,\varphi],f)$ be a $\bar{K}$-cycle (that is, a cycle as in Definition \ref{relKthCycMod}) and $V$ be an even rank $spin^c$-vector bundle over $M$. Also, let $\pi_V : V \rightarrow M$ and $\pi: M\times [0,1] \rightarrow M$ deonte the relevant projection maps. Then
\begin{eqnarray*}
\mu_{\phi}((M,[E,F,\varphi],f)^V) & = & \mu_{\phi}(M^V,[\pi_V^*(E), \pi_V^*(F),\pi_V^*(\varphi)]\otimes \beta ,f\circ \pi_V)  \\
& = & (M^V\times [0,1],[(\mathcal{\pi_V^*(E)}_{B_2},\mathcal{\pi_V^*(F)}_{B_1},\alpha_{\pi_V^*(\varphi)})],f\circ \pi_V)\\
& = & (M\times [0,1], [(\mathcal{E}_{B_2}, \mathcal{F}_{B_1}, \alpha)],f\circ \pi)^{\pi^*(V)} \\
& = & \mu_{\phi}(M,[E,F,\varphi],f)^{\pi^*(V)} \\
& \sim & \mu_{\phi}(M,[E,F,\varphi],f)
\end{eqnarray*}
For the bordism relation (with respect to $\bar{K}_*(X;\phi)$), suppose that the $\bar{K}$-cycle, $(M,[E,F,\varphi],f)$, is the boundary of $(\tilde{M},[\tilde{E},\tilde{F},\tilde{\varphi}],\tilde{f})$.  Using the straightening the angle technique, we have a smooth $spin^c$-manifold, $\tilde{Z}$, which is homeomorphic to $\tilde{M}\times [0,1]$; let $h$ denote such a homeomorphism.  By construction, $\partial \tilde{Z}$ contains $M\times [0,1]$ as a regular domain and 
\begin{equation} \label{twoCopiesBor}
\partial \tilde{Z} - (M\times (0,1)) = \tilde{M}\dot{\cup} -\tilde{M}
\end{equation}
We form a bordism (with respect to $K_*(X;\phi)$) by taking 
$$((\tilde{Z}, M\times[0,1]),(\mathcal{E}_{B_2},\mathcal{F}_{B_1},\alpha), g)$$
where 
\begin{enumerate}
\item $\pi:\tilde{M}\times [0,1] \rightarrow \tilde{M}$ is the projection map;
\item $\mathcal{E}_{B_2}=h^*(\pi^*(\tilde{E}))$;
\item $\mathcal{F}_{B_1}$ is given by $\tilde{E}$ on the first copy of $\tilde{M}$ and $\tilde{F}$ on the second (see Equation \ref{twoCopiesBor}); 
\item $\alpha$ is given by the identity on the fibers of the first copy of $\tilde{M}$ and $\varphi$ on the second;
\item $g=\tilde{f}\circ \pi \circ h$.
\end{enumerate}       
Thus, $\mu_{\phi}$ is well-defined as a map from $\bar{K}_*(X;\phi) \rightarrow K_{*+1}(X;\phi)$.  
\begin{theorem}
If $X$ is a finite CW-complex, then $\mu_{\phi}$ is an isomorphism.
\end{theorem}
\begin{proof}
The Five Lemma and Theorems \ref{exactSeq111} and \ref{bockTypeSeq} reduce the proof to showing the commutativity of the diagram: 
\begin{center}
$\minCDarrowwidth12pt\begin{CD}
@>>> K_1(X;B_2) @>\bar{r}>> \bar{K}_0(X;\phi) @>\bar{\delta}>> K_0(X;B_1) @>\phi_*>> K_0(X;B_2) @>\bar{r}>> \bar{K}_1(X;\phi) @>>>  \\
@. @| @V\mu_{\phi}VV @| @| @V\mu_{\phi}VV  @. \\
 @>>> K_1(X;B_2) @>r>> K_1(X;\phi) @>\delta>> K_0(X;B_1) @>\phi_*>> K_0(X;B_2) @>r>> K_0(X;\phi) @>>> 
\end{CD}$
\end{center}
\vspace{0.2cm}
where the maps are defined in Definition \ref{isoDefnModel1And2} and Theorems \ref{exactSeq111} and \ref{bockTypeSeq}. \par
The only nontrivial part is proving that $\mu_{\phi}\circ \bar{r} = r$.  Let $(M,E_{B_2},f)$ be a cycle in $K_*(X;B_2)$.  There are two steps in our proof.  Firstly, we show that $(\mu_{\phi}\circ \bar{r})(M,E_{B_2},f)$ is in the image of $r$ and then show it is equal to $r(M,E_{B_2},f)$.  For the first step, by the definition of $\bar{r}$ and \cite[Section 3.21]{Kar}, there exists $B_1$-module bundle, $F_{B_1}$ and unitary $\hat{u}: \pi^*(F_{B_1})\otimes_{\phi}B_2 \rightarrow \pi^*(F_{B_1})\otimes_{\phi}B_2$ such that $\hat{u}\sim u$ in $K^1(M\times S^1;B_2)$ and
\begin{eqnarray*}
\bar{r}(M,E_{B_2},f) & = & (M\times S^1, \hat{r}(u), f\circ \pi) \\
& \sim & (M\times S^1, (\pi^*(F_{B_1}),\pi^*(F_{B_1}),\hat{u}),f\circ \pi)
\end{eqnarray*}
where $\pi:M\times S^1\rightarrow M$ is the projection map. By the definition of $\mu_{\phi}$, 
\begin{equation*}
(\mu_{\phi}\circ \bar{r})(M,E_{B_2},f) =  (M\times S^1\times [0,1], (\tilde{\pi}^*(F_{B_1}\otimes_{\phi} B_2, \pi^*(F_{B_1})\dot{\cup} \pi^*(F_{B_1}),\alpha_{\hat{u}}),f \circ \tilde{\pi}) 
\end{equation*}
where $\alpha_{\hat{u}}$ is defined as in Definition \ref{isoDefnModel1And2}. Moreover, the bordism relation with respect to $K_*(X;\phi)$ implies that 
\begin{eqnarray*}
(M\times S^1\times [0,1], (\tilde{\pi}^*(F_{B_1}\otimes_{\phi} B_2), \pi^*(F_{B_1}),\alpha_{\hat{u}}),f \circ \tilde{\pi}) & \sim & (M\times S^1 \times S^1, (V_{\hat{u}},\emptyset, \emptyset),f\circ \hat{\pi}) \\
& = & r(M\times S^1 \times S^1, V_{\hat{u}},f\circ \hat{\pi})
\end{eqnarray*}
where 
\begin{enumerate}
\item $\tilde{\pi}:M\times S^1\times [0,1] \rightarrow M$ is the projection map;
\item $\hat{\pi}:M\times S^1\times S^1\rightarrow M$ is the projection map;
\item $V_{\hat{u}}$ is the $B_2$-module bundle obtained by clutching the vector bundle $\pi^*(F_{B_1})\otimes_{\phi}B_2$ along $M\times S^1$ via $\hat{u}$.  
\end{enumerate}
This completes the first step of the proof. To summarize, we have shown that 
$$(\mu_{\phi}\circ \bar{r})(M,E_{B_2},f)=r(M\times S^1 \times S^1, V_{\hat{u}},f\circ \hat{\pi})$$ 
\par
The second step is to show that $(M\times S^1 \times S^1, V_{\hat{u}},f\circ \hat{\pi}) \sim (M,E_{B_2},f)$.  Using the fact that $u$ and $\hat{u}$ are equivalent in $K^1(M\times S^1;B_2)$, the bordism relation, and vector bundle modification, we obtain  
\begin{eqnarray*}
(M\times S^1 \times S^1, V_{\hat{u}},f\circ \hat{\pi}) & \sim & (M\times S^1 \times S^1, V_u,f\circ \hat{\pi}) \\
& \sim_{bor} & (M\times S^2, (\pi^{\prime})^*(E_{B_2})\otimes F_{{\rm Bott}}, f\circ \pi^{\prime}) \\
& \sim_{VBM} & (M,E_{B_2},f) 
\end{eqnarray*}
where $\pi^{\prime}:M\times S^2 \rightarrow M$ is the projection map, $F_{{\rm Bott}}$ denotes the Bott bundle, and $V_u$ is defined in the same way as $V_{\hat{u}}$ was defined. We note that the explicit bordism used above is obtained from a bordism between $S^2$ and $S^1 \times S^1$ (each with an appropriate vector bundle). 
\end{proof}

\begin{ex} {\bf $K$-homology with coefficients} \label{coeExa} \vspace{0.1cm} \\
In this example, we discuss $K$-homology with coefficients in certain abelian groups. An introduction to $K$-theory and $K$-homology with coefficients in the abelian groups of interest here can be found in \cite[Section 23.15]{Bla} (also see \cite{AAS, APS2, APS3}); geometric K-homology with coefficients in $\zkz$ is the topic of \cite{Dee1, Dee2}. Given an abelian group, $G$, and a finite CW-complex, $X$, we denote the $K$-theory of $X$ with coefficients in $G$ by $K^*(X;G)$ and the $K$-homology of $X$ with coefficients in $G$ by $K_*(X;G)$. 
\par
The fundamental property of $K$-homology (or any generalized homology theory) with coefficients is the Bockstein sequence (see \cite{Dee1} in the case of $\zkz$). Suppose that 
$$0 \rightarrow G_1 \rightarrow G_2 \rightarrow G_3 \rightarrow 0$$ 
is a short exact sequence of abelian groups.  Then, the Bockstein sequence associated to this short exact sequence is the six-term exact sequence of $K$-homology with coefficients: 
\begin{center}
$\begin{CD}
K_0(X;G_1) @>\phi_*>> K_0(X;G_2) @>r_{G_3}>> K_0(X;G_3) \\
@AA\delta_{G_3} A @. @VV\delta_{G_3} V \\
K_1(X;G_3) @<r_{G_3}<<  K_1(X;G_2) @<\phi_*<< K_1(X;G_1) 
\end{CD}$
\end{center}
\vspace{0.2cm}
The mapping cone of a $*$-homomorphism, $\phi$ and the six-term exact sequence in $KK$-theory can be used to construct such sequences.  Prototypical examples of $\phi$ are given by the following unital inclusions 
\begin{equation*}
\field{C}  \rightarrow  M_n(\field{C}), \ 
\field{C}  \rightarrow  Q, \
\field{C}  \rightarrow  N
\end{equation*}
where $Q$ is a UHF-algebra with $K_0$-group, $\field{Q}$, and $N$ is a ${\rm II_1}$-factor (recall that $K_0(N)\cong \field{R}$ and $K_1(N)\cong \{0\}$).  In these cases, Theorems \ref{exactSeq111} and \ref{bockTypeSeq} produce the Bockstein sequence associated (respectively) to the following exact sequences of abelian groups: 
\begin{eqnarray*}
0 \rightarrow \field{Z} \rightarrow \field{Z}  \rightarrow & \zkz & \rightarrow 0 \\
0 \rightarrow \field{Z} \rightarrow \field{Q}  \rightarrow & \field{Q}/\field{Z} & \rightarrow 0 \\
0 \rightarrow \field{Z} \rightarrow \field{R}  \rightarrow & \field{R}/\field{Z} & \rightarrow 0 
\end{eqnarray*}
In other words, Theorems \ref{exactSeq111} and \ref{bockTypeSeq} produce geometric models (via the cycles in Definitions \ref{relKthCycMod} and \ref{phiCyc} respectively) for $K_*(X;\zkz)$, $K_*(X;\field{Q}/\field{Z})$, and $K_*(X;\field{R}/\field{Z})$.
\end{ex}
In the next section, we discuss geometric $K$-homology with coefficient in $\rz$ in detail.  Before doing so, we discuss a number of generalizations of the constructions considered to this point.
\par
The Baum-Douglas model has been generalized to the equivariant and families index settings (see \cite{BHSEquPaper, BOOSW, EM, Wal}).  In a number of cases, the models constructed in this paper also have such generalizations.  In the equivariant setting, in the case of compact Lie groups, one should replace our cycles with the natural analogue based on \cite{Wal}.  In the case of a discrete group which acts properly, one should replace our cycles with the natural analogue of the cycles in \cite{BHSEquPaper}.  A generalization to actions of groupoids also seems possible, but much more involved.  The interested reader should compare the cycles defined in \cite{BHSEquPaper} with those in \cite{EM} as a starting point.  
\par
Let $D$ denote a unital $C^*$-algebra.  Then a model for $KK(C(X),D\otimes C_{\phi})$ can be obtained directly from our results.  One simply notes that the main results in this paper can be applied to the $*$-homomorphism $id_D \otimes \phi:D\otimes B_1\rightarrow D\otimes B_2$.  If $D$ is a commutative $C^*$-algebra (i.e., $D=C(Y)$), then there is an alternative approach; the cycles in this theory are defined as follows.  
\begin{define} \label{YCycles}
A cycle (over $X$ with respect to $\phi$ and $Y$) is given by a triple, $(W, (E_{B_2},F_{B_1},\alpha),f)$, where 
\begin{enumerate}
\item $W$ is a smooth, compact $spin^c$-manifold with boundary;
\item $E_{B_2}$ is a finitely generated projective Hilbert $B_2$-module bundle over $W\times Y$;
\item $F_{B_1}$ is a finitely generated projective Hilbert $B_1$-module bundle over $\partial W\times Y$;
\item $\alpha : E_{B_2}|_{Y\times \partial W} \cong \phi_*(F_{B_1})$ is an isomorphism of Hilbert $B_2$-module bundles;
\item $f:W \rightarrow X$ is a continuous map.
\end{enumerate}
\end{define}
The relation on such cycles is the natural generalization of the relation discussed in Section \ref{BorModel}.  One can replace the bundle data in such a cycle with a class in $K^0(Y\times W,Y\times \partial W;\phi)$. We will make use of this model in the next section in the particular case of $K$-homology with coefficients in $\rz$. 

\section{$\rz$-valued index theory} \label{rzIndexSec}
In this section, we specialize to the case of $\phi$ is the unital inclusion of the complex number into a ${\rm II}_1$-factor. As discussed in Example \ref{coeExa}, for this choice of $\phi$, $K_*(X;\phi)$ is a realization of $K_*(X;\rz)$.  The reader should compare our construction in Section \ref{rzIndMap} to \cite[Section 5]{APS3} and the pairings in Section \ref{rzIndPai} to the development in \cite[Section 6]{HigRoeEta}. 
\par
As the reader may recall if $X$ is a finite CW-complex and $G$ is an abelian group, then $K^*(X;G)$ (respectively, $K_*(X;G)$) denotes the $K$-theory (respectively, $K$-homology) of $X$ with coefficients in $G$. The reader may find it useful to refer back to the following list of notation regarding the specific models of $K_*(X;G)$ and $K^*(X;G)$ when reading this section:
\begin{enumerate}
\item $K_*(X;\rz)$ denotes the realization of $K$-homology with coefficients in $\rz$ via (depending on context) cycles as in Definition \ref{cycBunDat} or Definition \ref{phiCyc}; the reader should note that the only difference between these cycles is the use of bundle data in Definition \ref{cycBunDat} and $K$-theory data in Definition \ref{phiCyc};
\item $K^*(X;\rz)$ denotes the realization of $K$-theory with coefficients in $\rz$ via cycles as in Definition \ref{YCycles};
\item $K^*(X;\field{R})$ (respectively, $K^*(X;\field{Q})$) denotes the realization of $K$-theory with real (respectively, rational) coefficients given by $K_*(pt;C(X)\otimes A)$ where $A$ is a ${\rm II}_1$-factor (respectively, a UHF-algebra with $K_0$ the rational numbers); 
\item $K^*(X;\zkz)$ denotes the realization of $K$-theory with coefficients in $\zkz$ via cycles as in \cite{Dee1} (i.e., using $\zkz$-manifold theory);
\item $K^*(X;\qz):=\lim  K^*(X;\zkz)$;
\item $K^*_{APS}(X;\rz):= {\rm coker}(q)$ where $q:K^*(X;\field{Q}) \rightarrow K^*(X;\qz)\oplus K^*(X;\field{R})$ is the natural map (for more see \cite{APS3} or Section \ref{rzIndMap} below);
\item $K^*_{Basu}(X;\rz)$ is the realization of $K$-theory with coefficients in $\rz$ constructed in \cite{Bas} (we use this model only in Sections 6.2 and 6.3);
\item $K^*_{Lott}(X;\rz)$ is the realization of $K$-theory with coefficients in $\rz$ constructed in \cite{Lot} (we use this model only in Section 6.3);
\end{enumerate}
\subsection{The index map and $K$-theory with $\rz$-coefficients} \label{rzIndMap}
Let $Y$ denote a compact Hausdorff space; a useful special case to consider is $Y=pt$. Following \cite[Section 5]{APS3}, we begin with the observation that $K^*(Y;\rz)\cong {\rm coker}(q)$ where $q: K^*(Y;\field{Q}) \rightarrow K^*(Y;\qz) \oplus K^*(Y;\field{R})$ is the natural map.  The group ${\rm coker}(q)$ is of course graded; we denote the grading via ${\rm coker}(q^*)$.
\par
The goal of this subsection is the construction of an explicit isomorphism at the level of cycles from $K^*(Y;\rz)$ (modeled using cycles in Definition \ref{YCycles}) to ${\rm coker}(q)$. When $Y=pt$, this amounts to the construction of a map from cycles in $K^0(pt;\rz)$ to $\rz$. In general, based on the definition of $q$, we must construct maps from cycles in $K^*(Y;\rz)$ to $K^*(Y;\qz)$ and $K^*(Y;\field{R})$ respectively.  We will refer to the map in the case of $Y=pt$ as an index map.  
\par
Let $(W,\xi)$ be a cycle in $K^*(Y;\rz)$ and $(\partial W, \xi_{\field{C}})$ denote $\delta(W,\xi)\in K^{*-1}(Y)$ (recall that $\delta$ was defined in the statement of Theorem \ref{bockTypeSeq}).  
\par
The construction of the map from $(W,\xi)$ to an element in $K^*(Y;\qz)$ is as follows.  Since $\phi_*(\partial W, \xi_{\field{C}})=0$, $\phi_*$ is rationally injective, and \cite[Corollary 4.5.16]{Rav}, there exists $k\in \field{N}$, normal bundle $N_W$ over $W$, and bordism (with respect to the geometric model of the group $K^{*-1}(Y)$), $(Q,\eta_{\field{C}})$, such that
$$ k(\partial W, \xi_{\field{C}})^N = \partial (Q, \eta_{\field{C}})$$
where $N$ denotes $N_W$ restricted to $\partial W$.  By construction, $Q$ has the structure of a $spin^c$ $\zkz$-manifold (see for example \cite{Dee1, FM}). By \cite{Dee1}, $(Q,\eta_{\field{C}})$ defines a cycle in $K^*(Y;\qz):= \lim K^*(Y;\zkz)$.  The reader should note that, when $Y=pt$, the Freed-Melrose index (see \cite{Dee1,Dee2,FM}) of this cycle produces the required element of $\qz$.  
\par       
The element of $K^*(Y;\field{R})$ (associated to $(W,\xi)$) is given by the  cycle:
\begin{equation} \label{RValuedCycle}
(kW,\frac{1}{k} \xi_{N})^{N_W} \cup_{k(\partial W)^{N}} (-Q, \frac{1}{k}\phi_*(\eta_{\field{C}})) 
\end{equation}
Notice that the choice of $K$-theory class in this cycle is not unique; it depends on a choice of clutching function, which we denote by $u$. The well-definedness of this construction will be discussed shortly. Again, when $Y=pt$, one takes the $\field{R}$-valued index of this cycle to produce the required element in $\field{R}$. 
\par
Let ${\rm ind}_{\rz}$ denote the map produced by these two constructions.  Our first goal is to prove that this map is well-defined.  The reader should note that the map is {\it not} well-defined as a map to $K^*(Y;\qz) \oplus K^*(Y;\field{R})$ because we have made a number of choices in the construction of the image of $(Z,\xi)$ under ${\rm ind}_{\rz}$.  \par
To summarize, these choices are as follows:
\begin{enumerate}
\item The $k\in \field{N}$;
\item The normal bundle $N_W$ over $W$;
\item The bordism $(Q,\eta_{\field{C}})$ (which gives a class in $K^*(Y;\qz)$);
\item The choice of clutching function $u$.
\end{enumerate}
As above, let $N$ denote the restriction of $N_W$ to $\partial W$.  Using the obvious notation, let $k^{\prime}$, $N^{\prime}_W$, $(Q^{\prime}, \eta^{\prime}_{\field{C}})$, and $u^{\prime}$ be different choices of the previously listed data.  By assumption, 
\begin{eqnarray*}
k(\partial W, \xi_{\field{C}})^N & = & \partial (Q,\eta_{\field{C}}) \\
k^{\prime}(\partial W, \xi_{\field{C}})^{N^{\prime}} & = & \partial (Q^{\prime},\eta^{\prime}_{\field{C}})
\end{eqnarray*}
To begin, the nature of the inductive limits used to define $K$-theory with coefficients in $\field{Q}$ and $\qz$ implies that we can assume $k=k^{\prime}$; in other words, one can replace $k$ and $k^{\prime}$ with $k\cdot k^{\prime}$.
\par
By Lemma \ref{norStaIso}, there exists trivial bundles, $\epsilon$ and $\epsilon^{\prime}$ over $W$, such that 
$$N_W \oplus \epsilon \cong N_{W^{\prime}}^{\prime} \oplus \epsilon $$
Lemma \ref{twoVecMod} implies that 
$$((W,\xi)^{N_W})^{p_1^*(\epsilon)} \sim_{bor} ((W,\xi)^{N_{W^{\prime}}^\prime})^{p_2^*(\epsilon^{\prime})} $$
Let $((Z,\tilde{W} \dot{\cup} \tilde{W}^{\prime}), \tilde{\xi})$ be a fixed choice of such a bordism; in particular, let $\tilde{W}$ (respectively, $\tilde{W}^{\prime}$) denote the manifold (with boundary) in the cycle $((W,\xi)^{N_W})^{p_1^*(\epsilon)}$ (respectively, $((W,\xi)^{N_{W^{\prime}}^\prime})^{p_2^*(\epsilon^{\prime})}$). Then $(\partial Z - {\rm int}(\tilde{W} \dot{\cup} \tilde{W}^{\prime}), \tilde{\xi}_{\field{C}})$ defines a bordism (with respect to the group $K^{*-1}(Y)$) between $((\partial W, \xi_{\field{C}})^N)^{p^*_1(\epsilon)}$ and $((\partial W, \xi_{\field{C}})^{N^{\prime}})^{p^*_2(\epsilon^{\prime})}$.  
\par
As such, we can form an element in $K^*(Y;\field{Q})$. Let
$$(Q,\eta_{\field{C}})^{\epsilon} \cup k(\partial Z - {\rm int}(\tilde{W} \dot{\cup} \tilde{W}^{\prime}), \tilde{\xi}_{\field{C}}) \cup (Q^{\prime},\eta^{\prime}_{\field{C}})^{\epsilon^{\prime}}$$
be the Baum-Douglas cycle formed by gluing along $\partial Z=\partial W \dot{\cup} -\partial W$ and using the clutching function $(u^{\prime})^{-1}\circ u$. Then consider this cycle as an element in $K^*(Y;\field{Q})$ using the inductive limit structure on $\field{Q}$.  To be more precise, this cycle is considered as a element in $K^*(Y;\field{Q})$ using the following commutative diagram:
\begin{center}
$\begin{CD}
K^*(Y) @>>> K^*(Y;\field{Q})   \\
 @VVV  @VVV  \\
  K^*(Y;\zkz) @>>> K^*(Y;\qz) 
\end{CD}$
\end{center}
where the vertical maps are the maps appearing in the relevant Bockstein sequneces for $K$-theory with coefficients in $\zkz$ and $\qz$ (respectively). \par
Let $\delta_{\zkz}$ denote the Bockstein map with respect to the coefficient group $\zkz$. Then, using the bordism and vector bundle modification relations defined in \cite{Dee1}, we have
\begin{eqnarray*}
(Q,\eta_{\field{C}}) & \sim & (Q,\eta_{\field{C}})^{\epsilon_{Q}} \\
& \sim & (Q,\eta_{\field{C}})^{\epsilon_{Q}} \dot{\cup} (-Q^{\prime},\eta_{\field{C}})^{\epsilon_{Q^{\prime}}} \dot{\cup} (Q^{\prime},\eta_{\field{C}})^{\epsilon_{Q^{\prime}}} \\
& \sim & \delta_{\zkz} \left( (Q,\eta_{\field{C}})^{\epsilon} \cup k(\partial Z - {\rm int}(\tilde{W} \dot{\cup} \tilde{W}^{\prime}), \tilde{\xi}_{\field{C}}) \cup (Q^{\prime},\eta^{\prime}_{\field{C}})^{\epsilon^{\prime}} \right) \dot{\cup} (Q^{\prime},\eta_{\field{C}})^{\epsilon_{Q^{\prime}}}\\
& \sim & \delta_{\zkz} \left( (Q,\eta_{\field{C}})^{\epsilon} \cup k(\partial Z - {\rm int}(\tilde{W} \dot{\cup} \tilde{W}^{\prime}), \tilde{\xi}_{\field{C}}) \cup (Q^{\prime},\eta^{\prime}_{\field{C}})^{\epsilon^{\prime}} \right) \dot{\cup} (Q^{\prime},\eta_{\field{C}})
\end{eqnarray*}
This equivalence and the commutative diagram discussed in the previous paragraph imply that the construction of the element of $K^*(Y;\qz)$ is unique up to the image of an element in the image of $q$.  The reader should note the specific element of $K^*(Y;\field{Q})$ is given by the cycle 
$$(Q,\eta_{\field{C}})^{\epsilon} \cup k(\partial Z - {\rm int}(\tilde{W} \dot{\cup} \tilde{W}^{\prime}), \tilde{\xi}_{\field{C}}) \cup (Q^{\prime},\eta^{\prime}_{\field{C}})^{\epsilon^{\prime}} $$
To complete the proof of well-definedness, we must show that the cycle $(kW,\frac{1}{k} \xi_{N}) \cup_{k\partial W} (-Q, \frac{1}{k}\phi_*(\eta_{\field{C}}))$ is equivalent to
$$  (kW,\frac{1}{k} \xi_{N}) \cup_{k\partial W} (-Q^{\prime}, \frac{1}{k}\phi_*(\eta^{\prime}_{\field{C}})) \dot{\cup} (Q,\eta_{\field{C}})^{\epsilon} \cup k(\partial Z - {\rm int}(\tilde{W} \dot{\cup} \tilde{W}^{\prime}), \tilde{\xi}_{\field{C}}) \cup (Q^{\prime},\eta^{\prime}_{\field{C}})^{\epsilon^{\prime}} $$
This follows from the bordism relation in $K^*(Y;\field{R})$ and the existence of the bordism $((Z,\tilde{W} \dot{\cup} \tilde{W}^{\prime}), \frac{1}{k}\cdot \phi_*(\tilde{\xi}))$.  Thus ${\rm ind}_{\rz}$ is well-defined.
\begin{prop}
The map ${\rm ind}_{\rz}: K^0(Y;\phi) \rightarrow {\rm coker}(q)$ is an isomorphism.
\end{prop}
\begin{proof}
This follows from the Five Lemma and the commutative diagram:
\begin{center}
$\minCDarrowwidth15pt\begin{CD}
@>>> K^1(Y;\rz) @>>> K^0(Y) @>>> K^0(Y;\field{R}) @>>> K^0(Y;\rz) @>>> K^1(Y) @>>>  \\
@. @VVV @VVV @VVV @VVV @VVV  @. \\
 @>>> {\rm coker}(q^1) @>>>  K^0(Y) @>>> K^0(Y;\field{R}) @>>> {\rm coker}(q^0) @>>> K^1(Y) @>>> 
\end{CD}$
\end{center}
where 
\begin{enumerate}
\item the horizontal maps are the Bockstein sequences with respect to the realizations of $K$-theory with coefficients in $\rz$;
\item the vertical maps are the identity map, except for the map from $K^*(Y:\rz)$ to ${\rm coker}(q^*)$; the definition of this map is given in the statement of the theorem.
\end{enumerate}
\end{proof}
Returning to the case when $Y$ is a point, the composition of ${\rm ind}_{\rz}$ with the map $(W,\xi,f) \rightarrow (W,\xi)$ gives an $\rz$-valued index map on $K_0(X;\rz)$.

\subsection{Index pairings and slant products} \label{rzIndPai}
A more detailed review of $K$-theory with coefficients in $\rz$ is required before discussing the various index pairings related to our theory.  In \cite{Bas}, a model for $K^*(X;\rz)$ is constructed using ideas of Karoubi; we denote the realization in \cite{Bas} by $K^*_{Basu}(X;\rz)$. Cocycles in $K^1_{Basu}(X;\rz)$ are given by triples, $(V_1,V_2,\varphi)$, where $V_i$ are vector bundle over $X$ and $\alpha: V_1\otimes_{\phi} N \rightarrow V_2\otimes_{\phi} N$ is an isomorphism of $N$-module bundles.  In other words, the $K$-theory of $X$ with coefficients in $\rz$ is given by $K^*(X;\phi)$ (as defined in Section \ref{KthKarRel}). The main of this subsection is the construction of the geometric slant product; the analytic slant product is a special case of the Kasparov product (see for example \cite[Exercise 9.8.9]{HR})
\begin{prop}
Let $X$ be a finite CW-complex and $Y$ be a compact Hausdorff space.  Then, we have well-defined slant products:
\begin{eqnarray*}  
K^q_{Basu}(Y\times X;\rz)\times K_p(X) & \rightarrow & K^{p+q}(Y;\rz) \\
K^q(Y\times X)\times K_p(X;\rz) & \rightarrow & K^{p+q}(Y;\rz) 
\end{eqnarray*}
In the case $Y=pt$, the slant product reduces to an index pairing; an element of $\rz$ is obtained by taking the index of cycle in $K_0(pt;\rz)$.
\end{prop}
We give a detailed treatment for the odd slant products; the even products are obtained using a similar construction. \par
For the slant product $K^1_{Basu}(Y\times X;\rz) \times K_1(X)$, fix a cocycle, $(V_1,V_2,\varphi) \in K^1_{Basu}(Y\times X;\rz)$, and a cycle, $(M,E,f) \in K_1(X)$.  Define the slant product at the level of cycles to be
$$(M\times [0,1], (\pi_M^*(E)\otimes (f^*(V_1)\otimes_{\phi}N),F,\alpha))$$
where
\begin{enumerate}
\item $\pi_M:M\times [0,1] \rightarrow M$ is the projection map;
\item $F$ is the vector bundle defined by $f^*(V_1)$ on $Y\times M\times \{0\}$ and $f^*(V_2)$ on $Y\times M\times \{1\}$;
\item $\alpha$ is the identity on the fibers at $Y\times M\times \{0\}$ and $\varphi$ on the fibers at $Y\times M\times \{1\}$;
\end{enumerate}
To show that this slant product is well-defined, we must show that it is invariant under the relations in both $K$-theory and $K$-homology. 
Starting with the relations in $K$-theory, suppose that $(V,V,\varphi)$ is an elementary cocycle and that $\varphi_t$ is a fixed homotopy from $Id_{\phi_*(V_1)}$ to $\varphi$.  Let $\field{D}$ denote the closed unit disk.  Then, the bordism 
$$ ((M \times \field{D}, M \times [0,1]),(\pi_{\field{D}}^*(f^*(V)\otimes E),\phi_*(\pi_{[0,1]}^*(f^*(V)\otimes E)),h_t))$$ 
has boundary given by the slant product of $(V,V,\varphi)$ with $(M,E,f)$; we note that $\pi_{[0,1]}$ and $\pi_{\field{D}}$ denote the (obvious) projection maps and $h_t$ is obtained from the homotopy $\phi_t$. Hence, the slant product vanishes for elementary cocycles. A similar construction implies that the construction respects isomorphism of cocycles. \par
Next, consider the relations in the $K$-homology group. The slant product clearly respects the disjoint union operation and relation.  For the bordism relation (with respect to $K_1(X)$), suppose that $(W,\bar{E},g)$ is a cycle with boundary, $(M,E,f)$.  By straightening the angle, we have a smooth compact $spin^c$-manifold (with boundary), $Z$, which is homeomorphic (via $h$) to $W\times [0,1]$.  Moreover, $\partial Z$ has a regular domain given by $M\times [0,1]$ and 
$$\partial Z - (M\times (0,1))=W\dot{\cup} -W$$  
Thus 
\begin{eqnarray*}
<(V_1,V_2,\varphi),(M,E,f)> & = & (M\times [0,1], (\pi_M^*(E)\otimes (f^*(V_1)\otimes_{\phi}N),F,\alpha)) \\
& = & \partial (Z,M\times[0,1],((\pi_W\circ h)^*(\bar{E})\otimes g^*(V_1) \otimes_{\phi} N ,\bar{F} ,\bar{\alpha} ))
\end{eqnarray*}
where 
\begin{enumerate}
\item $\pi_W:W\times [0,1] \rightarrow W$ is the projection map;
\item $\bar{F}$ is the vector bundle defined by $(\bar{E} \otimes f^*(V_1))$ on $W\times \{0\}$ and $(\bar{E} \otimes f^*(V_2))$ on $W\times \{1\}$;
\item $\bar{\alpha}$ is the identity on the fibers at $W\times \{0\}$ and $\varphi$ on the fibers at $W\times \{1\}$;
\end{enumerate}
Finally, for vector bundle modification, let $(V_1,V_2,\varphi)$ be a cocycle in $K^1_{Basu}(Y\times X;\rz)$, $(M,E,f)$ be a Baum-Douglas cycle and $V$ be an even rank $spin^c$-vector bundle over $M$.  Then, the well-definedness of the slant product follows since 
$$(M\times [0,1], (\pi_M^*(E)\otimes (V_1\otimes_{\phi}N),F,\alpha)) \sim (M\times [0,1], (\pi_M^*(E)\otimes (V_1\otimes_{\phi}N),F,\alpha))^{\pi^*(V)}$$ 
This completes the proof that the slant product $K^1_{Basu}(Y \times X;\rz)\times K_1(X) \rightarrow K^0(Y;\rz)$ is well-defined.
\par
For the slant product $K^1(Y\times X) \times K_1(X;\rz)$, fix a unitary, $u:Y\times X \rightarrow \mathcal{U}(n)$ (representing an element in $K^1(Y\times X)$), and a cycle, $(W,(E_N,E_{\field{C}},\alpha),f)$ (representing an element in $K_1(X;\rz)$) and define their slant product to be the cycle:
$$(W\times S^1, (\pi^*_W(E_N)\otimes (V_u\otimes_{\phi}N ),\pi^*_W(E_{\field{C}})\otimes V_u,\alpha \otimes id))$$
where
\begin{enumerate}
\item $\pi_W:W\times S^1$ is the projection onto $W$;
\item $V_u$ is the vector bundle obtained by clutching the trivial bundle $Y\times W\times [0,1] \times \field{C}^n$ using the automorphism associated to the function $u\circ (id_Y \times f): Y\times W \rightarrow \mathcal{U}(n)$.
\end{enumerate}
The proof that this slant product is well-defined is as follows.  Beginning with the relations in $K$-theory, suppose $u_t$ is a continuous path of unitaries.  For any $t$, $(\pi_W^*(E_{B_2})\otimes V_{u_i},\pi^*(F_{B_1})\otimes V_{u_i}, \alpha \otimes id)$ determines the same class in $K^0(Y\times W,Y\times \partial W;\phi)$. Since the slant product depends only on the $K$-theory class in $K^0(Y\times W,Y\times \partial W;\phi)$, it is independent of the particular unitary representive determining a class in $K^1(Y\times X)$.
\par
The relations in $K$-homology are considered next.  The proof for the disjoint union/direct sum relation is trivial.  For bordism, suppose that $(W,(E_N,E_{\field{C}},\alpha),f)$ is the boundarywith respect to the group $K^1(X;\rz)$) $((Q,W); let (E^{\prime}_N,E^{\prime}_{\field{C}},\alpha^{\prime}),g)$ denote such a bordism. We can form the bordism with respect to $K_0(Y;\rz)$
$$(Q\times S^1,W\times S^1),(\pi^*_Q(E_N)\otimes (V_u\otimes_{\phi}N ),\pi^*_{\partial Q-{\rm{int}(W)}}(E_{\field{C}})\otimes V_u,\alpha^{\prime} \otimes id) )$$
where the notation is as in the definition of the slant product (e.g., $\pi_Q:Q\times S^1 \rightarrow Q$ is the projection map). \par
Finally, for the vector bundle modification relation, let $V$ be a $spin^c$ vector bundle over $W$ of even rank.  Then the slant product of $[u]$ with $(W,(E_N,E_{\field{C}},\alpha),f)^V$ is given by 
$$(W\times S^1, (\pi^*_W(E_N)\otimes (V_u\otimes_{\phi}N ),\pi^*_W(F_{\field{C}})\otimes V_u,\alpha \otimes id))^{p^*(V)}$$
where $p:W\times S^1 \rightarrow W$. This completes the proof that the slant product, $K^1(Y\times X)\times K_1(X;\rz)\rightarrow K^0(Y;\rz)$ is well-defined.
\subsection{Relationship with Lott's pairing}
In \cite{Lot}, Lott discusses a model for $K$-theory with coefficients in $\rz$ for smooth manifolds using connections and differential forms.  Again, this construction is based on work of Karoubi.  In particular, Lott constructs a pairing between $K$-theory with coefficients in $\rz$ and $K$-homology which is given by the relative $\eta$-invariant.  The main goal of this section is a proof that the pairing defined in the previous section is equal to Lott's pairing.  
\par
In this section, we must restrict to the case when $X$ is a smooth manifold.  Let $K^*_{Lott}(X;\rz)$ denote the realization of $K$-theory with coefficients in $\rz$ discussed in \cite{Lot}.  We will only discuss the odd pairing in detail.  A cocycle in $K^1_{Lott}(X;\rz)$ is given by $((V_1,\nabla_1),(V_2,\nabla_2),\omega)$ where $V_1$ and $V_2$ are complex Hermitian vector bundles, $\nabla_1$ and $\nabla_2$ are Hermitian connections on $V_1$ and $V_2$ respectively, and $\omega \in \Omega^{odd}(M)/im(d)$ such that $d\omega =ch(\nabla_1)-ch(\nabla_2)$.  \par 
Recall (see \cite{Bas}) that the isomorphism from $K^1_{Basu}(X;\rz)$ to $K^1_{Lott}(X;\rz)$ is defined at the level of cocycles via
$$(V_1,V_2,\varphi) \mapsto ((V_1,\nabla_1),(V_2,\nabla_2),CS_{N}(\tilde{\nabla}_1,\varphi^*(\tilde{\nabla}_2))) $$
where 
\begin{enumerate}
\item $\nabla_1$ and $\nabla_2$ are connections on $V_1$ and $V_2$ respectively;
\item For $i=1$ or $2$, $\tilde{\nabla}_i:=\nabla_1 \otimes I + I \otimes d$ is a $N$-bundle connection on $V_i\otimes_{\phi}N$;
\item $CS_{N}(\nabla_1,\varphi^*(\nabla_2))$ is the Chern-Simon form associated to the $N$-bundle connections $\tilde{\nabla}_1$ and $\varphi^*(\tilde{\nabla}_2)$.
\end{enumerate}  
Further details on these cocycles and $K^*_{Lott}(X;\rz)$ can be found in \cite{Lot}.  The reader can find more details on the $\eta$-invariant in \cite{APS1,APS2,APS3}.  
\par
The pairing $K^1_{Lott}(X;\rz)\times K_1(X) \rightarrow \rz$ is given at the level of cocycle and cycle via 
\begin{eqnarray*}
<((V_1,\nabla_1),(V_2,\nabla_2),\omega),(M,E,f)> & := & \bar{\eta}(D_{f^*(\nabla_1)}) -\bar{\eta}(D_{f^*(\nabla_2)}) \\
& & - \int_M {\rm Todd}(M)\wedge ch(E) \wedge f^*(\omega) \hbox{ {\rm mod} }\field{Z} \end{eqnarray*}
where 
\begin{enumerate}
\item $f$ is assumed to be a smooth function; 
\item $D_{f^*(\nabla_i)}$ is the Dirac operator on $M$ twisted by $E\otimes f^*(V_i)$;
\item $\eta(D_{f^*(\nabla_i)})$ is the $\eta$-invariant associated to $D_{f^*(\nabla_i)}$;
\item $\bar{\eta}(D_{f^*(\nabla_i)})= \frac{\eta(D_{f^*(\nabla_i)})+\rm{dim(ker}(D_{f^*(\nabla_i)})}{2}$ mod $\field{Z}$.
\end{enumerate}
\begin{theorem}
Let $X$ be a smooth compact manifold.  Then, the following diagram commutes:
\begin{center}
$\begin{CD}
K^1_{Basu}(X;\rz)\times K_1(X) @>(\Phi\times Id)>> K^1_{Lott}(X;\rz)\times K_1(X) \\
@VVV  @VVV \\
K_0(pt;\rz) @>{\rm ind}_{\rz}>> \rz \\ 
\end{CD}$
\end{center}
where 
\begin{enumerate}
\item The vertical maps are the pairings;
\item $\Phi:K^1_{Basu}(X;\rz) \rightarrow K^1_{Lott}(X;\rz)$ is the isomorphism constructed in \cite{Bas} (see the introduction of this subsection);
\item ${\rm ind}_{\rz}$ is the index map defined in Section \ref{rzIndPai}.
\end{enumerate} 
In other words, the pairing defined in \cite{Lot} is equal to the pairing defined in Section \ref{rzIndPai}.
\end{theorem} 
\begin{proof}
Let $(V_1,V_2,\varphi)$ be a cocycle in $K^1_{Basu}(X;\rz)$, $(M,E,f)$ be a Baum-Douglas cycle in $K_1(X)$, and $((V_1,\nabla_1),(V_2,\nabla_2),CS_{N}(\tilde{\nabla}_1,\varphi^*(\tilde{\nabla}_2)))$ denote the image of $(V_1,V_2,\varphi)$ under the isomorphism from $K^1_{Basu}(X;\rz)$ to $K^1_{Lott}(X;\rz)$.  To streamline the proof, we introduce some notation.  Let $\eta_i$ denote $\bar{\eta}(D_{f^*(\nabla_i)})$ and $\mathfrak{F}$ denote the cycle $(M\times [0,1],(\pi^*(E\otimes f^*(V_1))\otimes_{\phi}N,F,\alpha)$; the reader should recall the notation of Section \ref{rzIndPai}. 
\par
In this notation, the proof amounts to showing the following equality: 
$${\rm ind}_{\rz}(\mathfrak{F})=\eta_1-\eta_2 - \int_M {\rm Todd}(M)\wedge ch(E) \wedge f^*(CS_{N}(\tilde{\nabla}_1,\varphi^*(\tilde{\nabla}_2))) $$
The computation of these indices requires us to fix quite a bit of data.  We complete the proof assuming that $(M,E,f)$ satisfies the property that $(M\times\{0\}\dot{\cup} M\times\{1\}, E \otimes f^*(V_1) \dot{\cup} E\otimes f^*(V_2))$ is a boundary as a Baum-Douglas cycle over a point.  The general case can be obtained using the model using $K$-theory classes and the existence of a normal bordism to a cycle with this property. Let
\begin{enumerate}
\item $(Q,G)$ be a particular choice of bounding cycle for $(M\times\{0\}\dot{\cup} M\times\{1\}, E \otimes f^*(V_1) \dot{\cup} 	E\otimes f^*(V_2))$.
\item $Z$ be the compact $spin^c$-manifold without boundary $Q\cup (M\times [0,1])$ and $(Z,F_N)$ be the cycle in $K_*(X;N)$ formed (via clutching the bundles using $\alpha$) from $(M\times [0,1],\pi^*(E)\otimes(f^*(V_1)))$ and $(Q,G)$;
\end{enumerate}
In addition, the Dirac operators associated to the various manifolds twisted by the appropriate vector or von Neumann bundle will be denoted using subscript notation (e.g., the Dirac operator on $Q$ will be denote by $D_Q$).  Then 
$${\rm ind}_{\rz}(\mathfrak{F})=(0,{\rm ind}_{\field{R}}(D_{Q\cup M \times [0,1]}))\in {\rm coker}(q)$$ 
where $q: \field{Q} \rightarrow \qz \oplus \field{R}$ is the natural map (see Section \ref{rzIndMap}).  Further geometric structures on the manifolds and bundles involved are required.  Without loss of generality (see \cite[Proposition 3]{Lot}), assume that $f$ is the identity map and $E$ is a trivial line bundle (this simplifies notation).
\begin{enumerate}
\item Let $g$ be a metric on $Z$ which is a product in a neighbourhood of $M\times [0,1] \subseteq Z$;
\item Fix a connection, $\nabla_F$ on $F$ which is compatible with $g$;
\item Fix an $N$-connection, $\nabla_N$ on $F_N$ which is compatible with $g$ and is equal to $\nabla_F \otimes I + I\otimes d$ on fibers over $Q \subseteq Z$.  Moreover, choose $\nabla_N$ so that it is of the form $\tilde{\nabla}_1 t + (1-t) \varphi^*(\tilde{\nabla})_2$ on $M\times [0,1]$ where 
$\tilde{\nabla}_i:= \nabla_i\otimes I + I \otimes d$;
\end{enumerate}
Using the specific geometric data above, we have
\begin{eqnarray*}
{\rm ind}_{N}(D_{Q\cup M\times [0,1]}) & = & \int_{Q\cup M \times [0,1]} {\rm Todd}(Z)ch_N(\nabla_N) \in \field{R} \\
& = & \int_{Q} {\rm Todd}(Q)ch_N(\nabla_N|_{Q}) \\ 
&  & + \int_{M\times [0,1]}{\rm Todd}(M\times[0,1])ch_N(\nabla_N|_{M\times [0,1]}) 
\end{eqnarray*}
The isomorphism from ${\rm coker}(q)$ to $\rz$ is given by taking the difference of the $\qz$-entry and $\field{R}$-entry mod $\field{Z}$.  Thus, 
$${\rm ind}_{\rz}(\mathfrak{F})= -{\rm ind}_{N}(D) \hbox{ {\rm mod} }\field{Z}$$
The fact that the formula for ${\rm ind}_N(D_{Q\cup M\times [0,1]})$ is local and the specific nature of the connections used lead to
$${\rm ind}_{\rz}(\mathfrak{F})=\eta_1 - \eta_2 - \int_{M\times [0,1]}{\rm Todd}(M\times [0,1])ch_N(\nabla_N|_{M\times [0,1]}) \hbox{ {\rm mod} }\field{Z}$$
The fact that $\nabla_N|_{M\times [0,1]}=t \tilde{\nabla}_1  + (1-t) \varphi^*(\tilde{\nabla}_2)$, the definition of the Chern-Simon form and the fact that the other differential forms involved are pullbacks further reduces this expression to 
$${\rm ind}_{\rz}(\mathfrak{F})=\eta_1 - \eta_2 - \int_{M}{\rm Todd}(M)CS_N(\tilde{\nabla_1},\varphi^*(\nabla_2)) \hbox{ {\rm mod} }\field{Z}$$
This completes the proof. 
\end{proof}
\vspace{0.25cm}  
{\bf Acknowledgments} \\
The author thanks Heath Emerson, Magnus Goffeng, Nigel Higson, Ralf Meyer and Thomas Schick for discussions and also thanks the referee for a number of useful comments. This work was supported by NSERC through a Postdoctoral fellowship.

\vspace{0.25cm}
Email address: rjdeeley@uni-math.gwdg.de \vspace{0.25cm} \\
{ \footnotesize MATHEMATISCHES INSTITUT, GEORG-AUGUST UNIVERSIT${\rm \ddot{A}}$T, BUNSENSTRASSE 3-5, 37073 G${\rm \ddot{O}}$TTINGEN, GERMANY}
\end{document}